\documentclass[10pt]{article}%

\usepackage{amsmath,amssymb,amsfonts}%
\usepackage{a4wide}%
\usepackage{theorem}%
\usepackage{epsfig,subfigure}%
\usepackage{color}%
\usepackage{hyperref}%

\theoremstyle{change}%

\newtheorem{definition}{Definition:}[section]%
\newtheorem{theorem}[definition]{Theorem:}%
\newtheorem{lemma}[definition]{Lemma:}%
\newtheorem{opt}{Optimization Problem}[section]%
{\theorembodyfont{\rmfamily} \newtheorem{remark}[definition]{Remark:}}%

\newenvironment{proof}
  {{\bf Proof:}}
  {\qquad \hspace*{\fill} $\Box$}%

\newcommand{\tm}{\times}%

\newcommand{\bx}{{\bf x}}

\newcommand{\diag}{\operatorname{diag}}

\newcommand{\co}{\operatorname{co}}

\newcommand{\n}{\nonumber}

\newcommand{\cD}{{\mathcal D}}

\newcommand{\cK}{{\mathcal K}}

\newcommand{\cT}{{\mathcal T}}

\newcommand{\T}{{\mathfrak S}}

\newcommand{\res}{\mathrm{res}}

\newcommand{\diam}{\mathrm{diam}}

\newcommand{\CC}{\mathcal{C}}%
\newcommand{\EC}{\mathcal{E}}%
\newcommand{\MC}{\mathcal{M}}%
\newcommand{\SC}{\mathcal{S}}%
\newcommand{\R}{\mathbb{R}}%
\newcommand{\Z}{\mathbb{Z}}%
\newcommand{\ep}{\varepsilon}%
\newcommand{\rmD}{\mathrm{D}}%
\newcommand{\rmd}{\mathrm{d}}%
\newcommand{\rme}{\mathrm{e}}%
\newcommand{\tp}{\mathrm{top}}%
\newcommand{\N}{\mathbb{N}}%
\newcommand{\CPA}{\operatorname{CPA}}
\newcommand{\trn}{^{\scriptscriptstyle \top}}%

\begin{document}

\title{Numerical approximation of the data-rate limit for state estimation under communication constraints}%
\author{Sigurdur Hafstein\footnote{Faculty of Physical Sciences, University of Iceland, Dunhagi 5, IS-107 Reykjavik, Iceland; e-mail: shafstein@hi.is}\ \ and Christoph Kawan\footnote{Fakult\"{a}t f\"{u}r Informatik und Mathematik, Universit\"{a}t Passau, Passau, Germany; e-mail: christoph.kawan@uni-passau.de}}%
\date{}%
\maketitle%

\begin{abstract}
In networked control, a fundamental problem is to determine the smallest capacity of a communication channel between a dynamical system and a controller above which a prescribed control objective can be achieved. Often, a preliminary task of the controller, before selecting the control input, is to estimate the state with a sufficient accuracy. For time-invariant systems, it has been shown that the smallest channel capacity $C_0$ above which the state can be estimated with an arbitrarily small error, depending on the precise formulation of the estimation objective, is given by the topological entropy or a quantity named restoration entropy, respectively. In this paper, we propose an algorithm that computes rigorous upper bounds of $C_0$, based on previous analytical estimates.%
\end{abstract}



%37B40   	Topological entropy
%93C10    Nonlinear systems
%93C41   	Problems with incomplete information

\section{Introduction}%

Networked control systems are spatially distributed systems, in which the communication between sensors, controllers and actuators is accomplished through a shared digital communication network. Examples can be found, for instance, in vehicle tracking, underwater communications for remotely controlled surveillance and rescue submarines, remote surgery, space exploration and aircraft design. Another large field of applications can be found in modern industrial systems, where industrial production is combined with information and communication technology (`Industry 4.0'). A fundamental problem in networked control is to determine the minimal requirements on the communication network for a specified control objective to be achieved.%

In the simplest model case, a sensor measures the states of a dynamical system at discrete sampling times and transmits the encoded state measurements through a finite-capacity channel to a controller at a remote location. In this framework, various works characterize or estimate the smallest channel capacity above which a given control objective (usually, stabilization of some sort) can be achieved. An even more fundamental problem is to determine the smallest capacity above which the controller is able to compute an estimate of the state with a given precision. This problem has been studied under various assumptions on the system and the channel. Notably, Savkin \cite{Sav} characterized the critical capacity by a quantity, which turns out to be infinite if the system is genuinely affected by noise, and otherwise reduces to the topological entropy of the system. A more recent contribution is Matveev and Pogromsky \cite{MP2}, which discusses three estimation objectives of increasing strength and provides constructive methods to obtain upper and lower bounds for the associated critical channel capacities. The contribution of the paper at hand consists in a numerical scheme to compute the upper bounds proposed in \cite{MP1,MP2}. Further studies about state estimation under communication constraints include \cite{KYu,LMi}.%

The systems studied in \cite{MP2} are of the form $x_{t+1} = \phi(x_t)$ with a $C^1$-map $\phi:\R^n \rightarrow \R^n$. The aim is to generate an accurate estimate $\hat{x}_t$ of the state $x_t$ at a remote location for initial states $x_0$ confined to a compact set $K \subset \R^n$. The only way to transmit information to the estimator is via a noiseless discrete channel. At each time instant $t$, a coder encodes $x_t$ by a symbol $e_t$ from a finite coding alphabet $\MC$. This process can be described by maps $\CC_t$ so that%
\begin{equation*}
  e_t = \CC_t(x_0,x_1,\ldots,x_t;\hat{x}_0,\delta),\quad \CC_t:(\R^n)^{t+1} \tm \R^n \tm \R_{>0} \rightarrow \MC,%
\end{equation*}
where $\hat{x}_0$ is an initial estimate satisfying $\|x_0 - \hat{x}_0\| \leq \delta$ for some $\delta>0$, depending on the aspired exactness the estimate. The estimation process similarly can be described by maps $\EC_t$ so that%
\begin{equation*}
  \hat{x}_t = \EC_t(e_0,e_1,\ldots,e_t;\hat{x}_0,\delta),\quad \EC_t:\MC^{t+1} \tm \R^n \tm \R_{>0} \rightarrow \R^n.%
\end{equation*}
To allow a certain flexibility in the transmission of information, the number of bits that can be transmitted in any time interval of length $r$ is not fixed, but confined between two numbers $b_-(r) \leq b_+(r)$, satisfying%
\begin{equation*}
  C := \lim_{r\rightarrow\infty}\frac{b_-(r)}{r} = \lim_{r\rightarrow\infty}\frac{b_+(r)}{r},%
\end{equation*}
where $C$ by definition is the capacity of the channel. A desirable objective is to obtain an estimate of the form $\|x_t - \hat{x}_t\| \leq \ep$ for all $t\geq0$, whenever $x_0,\hat{x}_0\in K$ and $\|x_0 - \hat{x}_0\| \leq \delta$, where $\delta = \delta(\ep)$. Writing $C_0$ for the smallest capacity $C$ above which this can be achieved for any $\ep>0$, it was shown in \cite{MP2} that $C_0 \geq h_{\tp}(\phi;K)$ and $C_0 = h_{\tp}(\phi;K)$ if $K$ is forward-invariant, where $h_{\tp}(\phi;K)$ is the topological entropy of $\phi$ on $K$.%

One problem with the estimation objective addressed above is that the gap between the initial error $\delta$ and the final error $\ep$ may be very large. Another problem is that a coding and estimation policy based on topological entropy is likely to suffer from a severe non-robustness, since topological entropy is highly discontinuous with respect to the dynamical system under consideration. To avoid a drastic degradation of accuracy and at the same time obtain a coding and estimation scheme that is more robust with respect to perturbations, one may require instead that $\|x_t - \hat{x}_t\| \leq G\delta$ for all $t\geq0$ with a constant $G>0$. The smallest channel capacity $C_0$ above which this objective can be achieved can be described in terms of a different entropy notion, introduced in \cite{MP3} under the name \emph{restoration entropy}. A closed-form expression for restoration entropy can be formulated in terms of the singular values of the linearized system. This expression, which at the same time is an upper bound on $h_{\tp}(\phi;K)$, has been derived earlier by the authors of \cite{MP2} in their papers \cite{MP1,PMa}, both for discrete- and continuous-time systems and also for time-varying systems.%

In this paper, we consider a flow $(\phi_t)_{t\in\R}$ generated by an ODE $\dot{x} = f(x)$ with a sufficiently smooth vector field $f$ on $\R^n$. Our analysis focuses on the dynamics of $(\phi_t)$ on a compact forward-invariant set $K$. Essentially following an approach used before for the computation of Lyapunov functions \cite{BGH,GH1,GH2,Mar}, we numerically compute a piecewise affine Riemannian metric on the simplices of a triangulation of $K$, which is then used to produce an upper estimate on $C_0$ in terms of the eigenvalues of the symmetrized derivative of the vector field $f$, computed with respect to that metric. Our algorithm works in two steps. The first one produces a piecewise affine function $P$ on the given triangulation with values in the positive definite $(n\tm n)$ symmetric matrices, designed in such a way to minimize the maximum of the largest generalized eigenvalue. The second step produces a Lyapunov-like function, which is used to scale the metric $P$ in order to make the largest generalized eigenvalue even smaller. It needs to be mentioned that the original estimate in \cite{MP2,PMa} involves not the largest generalized eigenvalue only, but the sum of the $k$ largest generalized eigenvalues, where $1 \leq k \leq n$ is chosen to maximize this sum. Hence, we can only expect good results in low dimensions, where typically $k=1$.%

We apply our algorithm to the well-known Lorenz system with standard parameters on a region containing the attractor. Using a simplified algorithm which works with a constant $P$, we are already able to improve former entropy estimates obtained in \cite{PMa} by analytical methods.%

The paper is organized as follows. In Section \ref{sec_prelim}, we recall the main result of \cite{MP2,PMa}, yielding upper estimates on the topological entropy and the critical channel capacity. Section \ref{sec_stateest} explains the relevance of our algorithm for the problem of state estimation under communications constraints. A detailed description of the algorithm is presented in Section \ref{sec_decription}. In Section \ref{sec_lorenz}, the example of the Lorenz system is discussed. Finally, Section \ref{sec_concl} contains some concluding remarks.%

\section{Preliminaries}\label{sec_prelim}

\paragraph{Notation.} We denote by $\Z$ the set of integers and write $\Z_+ = \{n\in\Z : n \geq 0\}$. We write $\SC_n$ for the space of $(n\tm n)$ real symmetric matrices and $\SC_n^+ \subset \SC_n$ for the space of all positive definite elements of $\SC_n$. If $\phi(t,x)$ denotes the (local) flow of an ODE $\dot{x}=f(x)$ in $\R^n$ and $v:\R^n \rightarrow E$ is a $C^1$-function into a Euclidean space $E$, we write $\dot{v}(x)$ for the \emph{orbital derivative} of $v$ at $x$, i.e.%
\begin{equation*}
  \dot{v}(x) = \frac{\rmd}{\rmd t}\Bigl|_{t=0}v(\phi(t,x)) \in L(\R,E) \cong E.%
\end{equation*}
By $I$ we denote the $(n\tm n)$ identity matrix for any $n\in\N$. By writing $A \succeq B$ for $A,B \in \SC_n$, we mean that $A - B$ is positive semi-definite. Furthermore, we write $B_{\ep}(x) = \{y \in \R^n : \|x - y\| < \ep\}$. Finally, we use the notation $i = 1:n$ as a short-cut for $i \in \{1,\ldots,n\}$.%

\paragraph{Upper bounds for topological entropy.} In the following, we recall the main result of \cite{PMa}, providing upper bounds on topological entropy and critical channel capacity. In \cite{PMa}, the result is formulated for nonautonomous ODEs. However, we only use the following autonomous version.%

Consider an ODE of the form%
\begin{equation}\label{eq_ode}
  \dot{x} = f(x),\quad x \in \R^n%
\end{equation}
with a $C^1$-vector field $f:\R^n \rightarrow \R^n$. Since we only consider solutions that evolve within a compact set, we may assume that all solutions are defined on the whole time domain. We write $\phi(t,x_0)$ for the unique solution satisfying the initial condition $x(0) = x_0$. For a fixed $t\in\R$, we also write $\phi_t:\R^n \rightarrow \R^n$ for the diffeomorphism $x \mapsto \phi(t,x)$. We further assume the existence of a compact forward-invariant set $K \subset \R^n$, i.e.\ $\phi_t(K) \subset K$ for all $t\geq0$.%

The topological entropy of $\phi$ on $K$, denoted by $h_{\tp}(\phi;K)$, can be defined as follows. For $\tau,\ep>0$, a subset $E \subset \R^n$ $(\tau,\ep)$-spans $K$ if for every $x\in K$ there is $y\in E$ with%
\begin{equation*}
  \max_{0\leq t\leq\tau}\|\phi(t,x) - \phi(t,y)\| \leq \ep.%
\end{equation*}
Writing $r(\tau,\ep,\phi,K)$ for the minimal cardinality of any $(\tau,\ep)$-spanning set for $K$,%
\begin{equation*}
  h_{\tp}(\phi;K) := \lim_{\ep\downarrow0}\limsup_{\tau\rightarrow\infty}\frac{1}{\tau}\log_2 r(\tau,\ep,\phi,K).%
\end{equation*}

\begin{theorem}\label{thm_matpog}
Let $P:K \rightarrow \SC_n^+$ and $v_d:K \rightarrow \R$, $1 \leq d \leq n$, be $C^1$-functions and let $\lambda_1(x) \geq \ldots \geq \lambda_n(x)$ denote the solutions of the algebraic equation%
\begin{equation}\label{eq_peq}
  \det\left[ \rmD f(x)\trn P(x) + P(x) \rmD f(x) + \dot{P}(x) - \lambda P(x) \right] = 0.%
\end{equation}
Let $\Lambda_d\geq0$, $1 \leq d \leq n$, be constants so that%
\begin{equation}\label{eq_lyaplikeineq}
  \sum_{i=1}^d \lambda_i(x) + \dot{v}_d(x) \leq \Lambda_d \mbox{\quad for all\ } x\in K.%
\end{equation}
Then for $\Lambda := \max_{1\leq d\leq n}\Lambda_d$, the topological entropy of $\phi$ on $K$ satisfies%
\begin{equation*}
  h_{\tp}(\phi;K) \leq \frac{\Lambda}{2\ln 2}.%
\end{equation*}
\end{theorem}

Some remarks about the formulation of the theorem are in order.%

\begin{remark}
The functions $P$ and $v_d$ in \cite{PMa} depend on three variables, i.e.\ $P = P(t,s,x_0)$ and $v_d = v_d(t,s,x_0)$, where $t \geq s$ are time variables. Such functions can be obtained from the above formulation by putting%
\begin{equation*}
  \tilde{P}(t,s,x_0) := P(\phi(t-s,x_0)),\quad \tilde{v}_d(t,s,x_0) := v_d(\phi(t-s,x_0)),%
\end{equation*}
and it is easy to verify that the so-defined functions satisfy the requirements of \cite[Thm.~3.2]{PMa}.%
\end{remark}

\begin{remark}\label{rem2}
The function $P$ can be interpreted as a Riemannian metric on $K$, defined by%
\begin{equation*}
  \langle v,w \rangle_x := \langle P(x)v,w \rangle \mbox{\quad for all\ } x\in K.%
\end{equation*}
Indeed, let $X(\cdot)$ denote the solution to the following initial value problem corresponding to the variational equation of \eqref{eq_ode}:%
\begin{equation*}
  \dot{Y}(t) = \rmD f(\phi_t(x_0))Y(t),\quad Y(0) = I.%
\end{equation*}
Let $\alpha_1(t) \geq \ldots \geq \alpha_n(t)$ denote the singular values of $X(t)$ w.r.t.~the metric $\langle\cdot,\cdot\rangle_{(\cdot)}$, i.e.\ the eigenvalues of the self-adjoint operator $\sqrt{X(t)^*X(t)}$, where $X(t)^*$ is defined by $\langle X(t)v,w \rangle_{\phi_t(x_0)} \equiv \langle v,X(t)^*w \rangle_{x_0}$. Then, according to \cite[Prop.~8.6]{PMa},%
\begin{equation}\label{eq_singvalineq}
  \alpha_1(t)\alpha_2(t) \cdots \alpha_d(t) \leq \exp\left( \frac{1}{2}\int_0^t [\lambda_1(\phi(s,x_0)) + \cdots + \lambda_d(\phi(s,x_0))] \rmd s\right),\quad 1 \leq d \leq n.%
\end{equation}
We expect that the number $d$, where the maximum $\max_{1\leq d\leq n}\Lambda_d$ is attained, is more or less fixed for a given system under any reasonable choice of the functions $v_1,\ldots,v_n$ (it is something like the number of positive Lyapunov exponents). If this is the case, we can also incorporate the function $v_d$ into the metric by putting%
\begin{equation*}
  \langle v,w \rangle_x := \langle \rme^{v_d(x)/d}P(x)v,w \rangle \mbox{\quad for all\ } x\in K.%
\end{equation*}
Then \eqref{eq_peq} is equivalent to%
\begin{equation*}
  \det\left[ \rmD f(x)\trn P(x) + P(x) \rmD f(x) + \dot{P}(x) - \left(\lambda - \frac{1}{d}\dot{v}_d(x)\right) P(x) \right] = 0,%
\end{equation*}
and thus, the sum $\sum_{i=1}^d \lambda_i(x)$ with the solutions of \eqref{eq_peq} becomes $\sum_{i=1}^d \lambda_i(x) + \dot{v}_d(x)$.%
\end{remark}

The functions $v_i$ in Theorem \ref{thm_matpog} have some similarity with Lyapunov functions. Instead of $\dot{v}_i < 0$ we have the inequalities \eqref{eq_lyaplikeineq}. In the rest of the paper, we call such functions \emph{Lyapunov-type functions}.%

\section{The state estimation problem}\label{sec_stateest}

In this section, we show how Theorem \ref{thm_matpog} is related to the problem of state estimation over a digital channel.%

Consider a dynamical system given by an ODE of the form \eqref{eq_ode}. Suppose that a sensor, fully observing the state $x_t$ of the system, sends its data to an encoder. At discrete sampling times, the encoder sends a signal $e_t$ through a noisefree discrete channel to a decoder (without transmission delay). The decoder acts as an observer of the system, trying to reconstruct the state from the received data. For simplicity, we assume that the times of transmissions are $t = 0,1,2,\ldots$. We write $x_t$ for the state at time $t$ and $\hat{x}_t$ for its estimate generated by the observer. Moreover, we assume that $x_0,\hat{x}_0 \in K$ for a compact and forward-invariant set $K \subset \R^n$. The encoder and the observer are described by mappings%
\begin{equation*}
  e_t = \CC_t(x_0,x_1,\ldots,x_t;\hat{x}_0,\delta),\quad \CC_t:(\R^n)^{t+1} \tm \R^n \tm \R_{>0} \rightarrow \MC,%
\end{equation*}
and%
\begin{equation*}
  \hat{x}_t = \EC_t(e_0,e_1,\ldots,e_t;\hat{x}_0,\delta),\quad \EC_t:\MC^{t+1} \tm \R^n \tm \R_{>0} \rightarrow \R^n.%
\end{equation*}
The argument $\delta$ corresponds to the initial error at time zero, i.e.\ $\|x_0 - \hat{x}_0\| \leq \delta$. In particular, we assume that both the encoder and the observer are given the data $\hat{x}_0$ and $\delta$.%

We assume that the channel can transmit at least $b_-(r)$ and at most $b_+(r)$ bits in any time interval of length $r$. The \emph{capacity} of the channel is then defined by%
\begin{equation*}
  C := \lim_{r\rightarrow\infty}\frac{b_-(r)}{r} = \lim_{r\rightarrow\infty}\frac{b_+(r)}{r},%
\end{equation*}
assuming that these limits exist and coincide.%

We consider the following two observation objectives:%
\begin{enumerate}
\item[(O1)] The observer \emph{observes} the system with exactness $\ep>0$ if there exists $\delta = \delta(\ep,K)$ so that $x_0,\hat{x}_0 \in K$ with $\|x_0 - \hat{x}_0\| \leq \delta$ implies%
\begin{equation*}
  \sup_{t\geq0}\|x_t - \hat{x}_t\| \leq \ep.%
\end{equation*}
\item[(O2)] The observer \emph{regularly observes} the system if there exist $G,\delta_*>0$ so that for all $\delta \in (0,\delta_*)$ and $x_0,\hat{x}_0 \in K$ with $\|x_0 - \hat{x}_0\| \leq \delta$,%
\begin{equation*}
  \sup_{t\geq0}\|x_t - \hat{x}_t\| \leq G\delta.%
\end{equation*}
\end{enumerate}

We say that the system is%
\begin{itemize}
\item \emph{observable on $K$} over a channel of capacity $C$ if for every $\ep>0$ an observer exists which observes the system with exactness $\ep$ over this channel;%
\item \emph{regularly observable on $K$} over a channel of capacity $C$ if there exists an observer which regularly observes the system over this channel.%
\end{itemize}

For objective (O1) we have the following result, cf.~\cite{MP2}:%

\begin{theorem}
The smallest channel capacity $C_0$, so that system \eqref{eq_ode} is observable on $K$ over every channel of capacity $C > C_0$ is given by%
\begin{equation*}
  C_0 = h_{\tp}(\phi;K).%
\end{equation*}
\end{theorem}

Due to the problems related to estimation policies based on topological entropy, and the gap between the initial error $\delta$ and the final exactness $\ep$, both explained in the introduction, \cite{MP3} introduces another entropy notion tailored to characterize $C_0$ for objective (O2).%

For $t > 0$, $x \in K$ and $\delta>0$ let $p(t,x,\delta)$ denote the minimal number of $\delta$-balls needed to cover the image $\phi_t(B_{\delta}(x) \cap K)$. The \emph{restoration entropy} of $\phi$ on $K$ is given by%
\begin{equation*}
  h_{\res}(\phi;K) := \lim_{t \rightarrow \infty}\frac{1}{t}\limsup_{\delta\downarrow0}\sup_{x \in K}\log_2 p(t,x,\delta).%
\end{equation*}
The limit in $t$ exists due to subadditivity, and the following data-rate theorem holds, cf.~\cite{MP3}.%

\begin{theorem}
The smallest channel capacity $C_0$, so that system \eqref{eq_ode} is regularly observable on $K$ over every channel of capacity $C > C_0$ is given by%
\begin{equation*}
  C_0 = h_{\res}(\phi;K).%
\end{equation*}
\end{theorem}

Now, $h_{\res}(\phi;K)$ is a quantity that is much better behaved than $h_{\tp}(\phi;K)$ in several respects. A first manifestation of this is the following characterization of $h_{\res}(\phi;K)$ in terms of the singular values of the derivative $\rmD \phi_t(x)$, cf.~\cite[Thm.~11]{MP3}:%

\begin{theorem}\label{thm_hresform}
Assume that the closure of $K$ equals the closure of its interior. Then%
\begin{equation}\label{eq_hresform}
  h_{\res}(\phi;K) = \lim_{t \rightarrow \infty}\frac{1}{t}\max_{x\in K}\sum_{i=1}^n \max\{0,\log_2 \alpha_i(t,x)\},%
\end{equation}
where $\alpha_1(t,x) \geq \ldots \geq \alpha_n(t,x)$ are the singular values of $\rmD \phi_t(x)$.%
\end{theorem}

The existence of the limit in \eqref{eq_hresform} follows again from subadditivity. Hence, the limit can be replaced by the infimum over all $t>0$. From this fact, one easily sees that $h_{\res}$ depends upper semicontinuously on the system under consideration (in the $C^1$-topology).%

We claim that Theorem \ref{thm_matpog} also holds with $h_{\res}(\phi;K)$ in place of $h_{\tp}(\phi;K)$. A heuristic argument, neglecting the functions $v_1,\ldots,v_n$, proceeds as follows. First, one shows that formula \eqref{eq_hresform} also holds if we compute the singular values of $\rmD\phi_t(x)$ with respect to some Riemannian metric on $K$, described by a $C^1$-function $P:K \rightarrow \SC_n^+$. The adjoint of $\rmD\phi_t(x)$ is then given by%
\begin{equation*}
  \rmD\phi_t(x)^* = P(x)^{-1}\rmD\phi_t(x)\trn P(\phi_t(x)),%
\end{equation*}
hence the singular value equation can be written as%
\begin{equation}\label{eq_gensingvalseq}
  \det\left[\rmD\phi_t(x)\trn P(\phi_t(x)) \rmD\phi_t(x) - \lambda P(x) \right] = 0.%
\end{equation}
Thus, due to subadditivity, for every $t>0$ we have%
\begin{equation}\label{eq_hres_allt}
  h_{\res}(\phi;K) \leq \frac{1}{t}\max_{x\in K}\sum_{i=1}^n \max\{0,\log_2 \alpha_i^P(t,x)\},%
\end{equation}
where $\alpha_1^P(t,x) \geq \ldots \geq \alpha_n^P(t,x)$ are the square-roots of the solutions to \eqref{eq_gensingvalseq}. Assuming the existence of differentiable curves $\lambda:[0,\ep) \rightarrow \R$ and $v:[0,\ep) \rightarrow \R^n$ with $\|v(t)\| \equiv 1$ so that%
\begin{equation}\label{eq_gsvsol}
  \rmD\phi_t(x)\trn P(\phi_t(x)) \rmD\phi_t(x) v(t) = \lambda(t) P(x)v(t) \mbox{\quad for all\ } t \in [0,\ep),%
\end{equation}
differentiation with respect to $t$ at $t = 0$ yields%
\begin{equation*}
  \bigl(\rmD f(x)\trn P(x) + P(x)\rmD f(x) + \dot{P}(x)\bigr)v(0) + P(x) \dot{v}(0) = \dot{\lambda}(0)P(x)v(0) + \lambda(0)P(x)\dot{v}(0).%
\end{equation*}
For $t=0$, equation \eqref{eq_gsvsol} reduces to $P(x) v(0) = \lambda(0)P(x)v(0)$, hence $\lambda(0)=1$. Consequently, the above equation is equivalent to%
\begin{equation*}
  \bigl(\rmD f(x)\trn P(x) + P(x)\rmD f(x) + \dot{P}(x) - \dot{\lambda}(0)P(x)\bigr)v(0) = 0.%
\end{equation*}
Letting $t \rightarrow 0$ in the right-hand side of \eqref{eq_hres_allt} and comparing with Theorem \ref{thm_matpog} then yields the claim. For a precise formulation and proof, we refer to \cite[Thm.~14]{MP3}.%

Hence, we can conclude that Theorem \ref{thm_matpog}, and thus our algorithm yields upper bounds for $h_{\res}(\phi;K)$, i.e.\ for the smallest channel capacity above which the estimation objective (O2) can be achieved. Moreover, the output of our algorithm can be used to implement a coding and estimation policy which leads to regular observation, as is shown in \cite{MP2,MP3}.%

\section{Description of the algorithm}\label{sec_decription}

In this section, we describe the algorithm for computing the upper bounds provided by Theorem \ref{thm_matpog}, which is split into two optimization problems. Before we go into details, we provide a short outline: Our algorithm aims at the computation of optimal functions $P$ and $v_i$ by solving two optimization problems. Starting with a triangulation $\cT$ of the compact forward-invariant set $K$ (or some larger set), the first optimization problem delivers a piecewise affine function $P$ on $K$, affine on each simplex of the triangulation $\cT$, with values in $\SC^+_n$. This is accomplished by solving a semidefinite feasibility problem with linear matrix inequality constraints at each vertex and extension to the whole domain by affine interpolation of the values obtained at the vertices. The decisive quantity in this problem is a positive parameter $\mu$, so that the solution $P$ (if it exists) satisfies $\lambda_{\max}(x) \leq \mu$ for all $x\in K$, where $\lambda_{\max}(x)$ denotes the largest generalized eigenvalue of the pair $(A(x),P(x))$ with%
\begin{equation*}
  A(x) := P(x)\rmD f(x) + \rmD f(x)\trn P(x) + (w_{ij}^\nu \cdot f(x))_{i,j=1:n},%
\end{equation*}
where $w_{ij}^{\nu}$ stands for the gradient of the $(i,j)$-th entry of $P$ on the simplex $\T_{\nu}$ satisfying $x \in \T_{\nu}$. If the algorithm yields a feasible solution for one parameter $\mu_1$, it can be run again for a smaller parameter $\mu_2 < \mu_1$ to check if there is still a feasible solution. Repeating this procedure, the maximum of the largest generalized eigenvalues over $K$ can be minimized.%

The second optimization problems takes as an input a feasible solution $P$ of the first problem and an upper bound $\widetilde{m}$ on the number of positive generalized eigenvalues of the matrix pairs $(A(x),P(x))$. It delivers a piecewise affine real-valued function $V$, affine on each simplex of a triangulation $\cT^*$, which is a refinement of $\cT$, and another real-valued function $\mu$. This is done by solving a semidefinite problem with linear matrix inequality constraints at each vertex of $\cT^*$ and extending again by affine interpolation. The optimization minimizes $Q$ so that for all $x\in K$,%
\begin{equation*}
  A(x) - \mu(x)P(x) \preceq 0 \mbox{\quad and \quad} \dot{V}(x) + \widetilde{m}\mu(x) \leq Q.%
\end{equation*}
A detailed description of these two steps is given in the following subsections. The main results are Theorem \ref{thm_amr1} and Theorem \ref{X4.12}, which show that solutions to the optimization problems, computed at the vertices of the triangulation, extend to solutions on the whole domain of interest by affine interpolation.%

\subsection{The semidefinite optimization problem}

Given vectors $x_0,x_1,\ldots,x_n\in\R^n$ that are affinely independent, i.e.\ the vectors $x_1-x_0,x_2-x_0,\ldots,x_n-x_0$ are linearly independent, the convex hull%
\begin{equation*}
  \T = \co(x_0,x_1,\ldots,x_n) := \left\{\sum_{k=0}^n \lambda_k x_k\,:\, \lambda_k\in[0,1]\ \text{and}\ \sum_{k=0}^n\lambda_k=1\right\}%
\end{equation*}
is called an $n$-simplex or simply a simplex. A set%
\begin{equation*}
  \co(x_{k_0},x_{k_1},\ldots,x_{k_j}) := \left\{\sum_{i=0}^j \lambda_{k_i} x_{k_i}\,:\, \lambda_{k_i}\in[0,1]\ \text{and}\ \sum_{i=0}^j\lambda_{k_i}=1\right\}%
\end{equation*}
with $0\le k_0< k_1 < \ldots <k_j \le n$ and $0\le j <n$ is called a $j$-face of the simplex $\T$.%

\begin{definition}[Triangulation]\label{scdef}
We call a finite set $\cT=\{\T_\nu\}_\nu$ of $n$-simplices $\T_\nu$ a triangulation in $\R^n$ if two simplices $\T_\nu,\T_\mu\in\cT$, $\mu\neq \nu$, intersect in a common face or not at all and the interior $\cD_\cT^\circ$ of $\cD_\cT:=\bigcup_{\T_\nu\in\cT} \T_\nu$ is connected.
\end{definition}

\begin{opt}\label{SDP1}
Given is a system $\dot x=f(x)$, $f\in C^3(\R^n;\R^n)$, a triangulation $\cT$ in $\R^n$, and a parameter $\mu\ge 0$. The optimization problem is a semidefinite feasibility problem with linear matrix inequality constraints.%

{\bf Constants:}
The constants used in this problem are%
\begin{enumerate}
\item $\epsilon_0>0$ -- lower bound on the matrix $P(x_k)$%
\item The diameter $h_\nu$ of each simplex $\T_\nu\in \cT${\rm :}
\begin{equation*}
  h_\nu :=\diam (\T_\nu)=\max_{x,y\in \T_\nu}\|x-y\|_2%
\end{equation*}

\item Upper bounds $B_\nu$ on the second-order derivatives of the components $f_k$ of $f$ on each simplex $\T_\nu\in\cT${\rm :}
\begin{equation}\label{Bdef}
  B_{\nu} \ge \max_{x\in \T_\nu \atop i,j,k=1:n }\left|\frac{\partial^2 f_k(x)}{\partial x_i\partial x_j}\right|%
\end{equation}
\item Upper bounds $B_{3,\nu}$ on the third-order derivatives of the components $f_k$ of $f$ on each simplex $\T_\nu\in\cT${\rm :}
\begin{equation*}
  B_{3,\nu} \ge \max_{x\in \T_\nu \atop i,j,k,l=1:n }\left|\frac{\partial^3 f_l(x)}{\partial x_i\partial x_j\partial x_k}\right|
\end{equation*}
\end{enumerate}

{\bf Variables:}
The variables of the problem are%
\begin{enumerate}
\item $P_{ij}(x_k)\in \mathbb R$ for all $1\le i\le j\le n$ and all vertices $x_k$ of all simplices $\T_\nu=\co(x_0,\ldots,x_{n})\in \cT$.  For $1\le i\le j\le n$ the variable $P_{ij}(x_k)$ is the $(i,j)$-th entry of the $(n\times n)$ matrix $P(x_k)$. The matrix $P(x_k)$ is assumed to be symmetric and therefore these components determine it.%
\item $C_\nu \in \mathbb R_0^+$ for all simplices $\T_\nu\in \cT$ -- upper bound on $P$ in $\T_\nu$%
\item $D_\nu \in \mathbb R_0^+$ for all simplices $\T_\nu\in \cT$ -- upper bound on the derivative of $P_{ij}$ in $\T_\nu$%
\end{enumerate}

{\bf Objective:}
The objective function of the optimization problem is not needed because it is a feasibility problem, but one can, e.g., minimize $\displaystyle \max_{\T_\nu\in\cT}C_\nu$.\\

{\bf Constraints:}

\begin{enumerate}
\item {\bf Positive definiteness of $\mathbf P$}

For each simplex $\T_\nu = \co(x_0,\ldots,x_n)\in\cT$ and each vertex $x_k$ of $\T_\nu$\,{\rm :}%
\begin{equation*}
  P(x_k)\succeq  \epsilon_0 I%
\end{equation*}

\item {\bf Upper bound on $\mathbf P$}

For each simplex $\T_\nu = \co(x_0,\ldots,x_n)\in\cT$ and each vertex $x_k$ of $\T_\nu$\,{\rm :}%
\begin{equation*}
  P(x_k)\preceq C_\nu I%
\end{equation*}

\item {\bf Bound on the derivative of $\mathbf P$}

For each simplex $\T_\nu\in\cT$ and all $1\le i\le j\le n$\,{\rm :}%
\begin{equation*}
  \|w^{\nu}_{ij} \|_1 \le D_{\nu}%
\end{equation*}
Here $w^{\nu}_{ij} = \nabla P_{ij}\big|_{\T_\nu}(x)$ for all $x\in \T_\nu$. See Remark \ref{nugradrem} for details.%

\item {\bf Bounds on the largest generalized eigenvalue}

For each simplex $\T_\nu = \co(x_0,\ldots,x_n)\in\cT$ and each vertex $x_k$ of $\T_\nu$\,{\rm :}%
\begin{equation*}
  0 \succeq A(x_k) - \mu P(x_k) + h_\nu^2  E_{\nu}I
\end{equation*}
Here%
\begin{equation*}
  A(x_k) := P(x_k)\rmD f(x_k) + \rmD f(x_k)\trn P(x_k) + (w_{ij}^\nu \cdot f(x_k))_{i,j=1:n},%
\end{equation*}
where $\rmD f(x_k)$ is the Jacobian matrix of $f$ at $x_k$, $(w_{ij}^\nu \cdot f(x_k))_{i,j=1:n}$ denotes the symmetric $(n\times n)$-matrix with entries $w_{ij}^\nu \cdot f(x_k)$ and $w^{\nu}_{ij}$ is defined as in \eqref{defw} and is the same vector for all vertices in one simplex. Further,%
\begin{equation*}
  E_{\nu} := n^2[(1+4\sqrt{n}) B_\nu D_\nu+2 n B_{3,\nu}C_\nu].%
\end{equation*}
\end{enumerate}
\end{opt}

\begin{remark}\label{nugradrem}
In Constraints 3 and 4 in Optimization Problem \ref{SDP1}, the gradient $w^\nu_{ij}$ of the affine function $P_{ij}\big|_{\T_\nu}$ on the simplex $\T_\nu=\co(x_0,\ldots,x_n)$, i.e.\ $\nabla P_{ij}\big|_{\T_\nu} = w_{ij}^\nu$, is given by the expression%
\begin{equation}\label{defw}
  w^{\nu}_{ij} := X^{-1}_{\nu}\left(\begin{array}{c}P_{ij}(x_1) - P_{ij}(x_0)\\ \vdots\\ P_{ij}(x_n) - P_{ij}(x_0)\end{array}\right)\in\mathbb R^n,%
\end{equation}
where $X_{\nu}=\left(x_1-x_0,x_2-x_0,\ldots,x_n-x_0\right)\trn\in\mathbb R^{n\times n}$ is the so-called shape-matrix of the simplex $\T_\nu$. For a proof of this fact and, moreover, that the definition is independent of the choice of the vertex $x_0$, see \cite[Rem.~2.9]{GH2}.%

The Constraints 3 are indeed linear and can be implemented using the auxiliary variables $D_\nu^{k}$ and the constraints%
\begin{equation*}
  -D_\nu^{k}\le [w^\nu_{ij}]_k \le D_\nu^{k}\ \ \text{for $k=1:n$},%
\end{equation*}
where $[w^\nu_{ij}]_k $ is the $k$-th component of the vector $w^\nu_{ij}$, and setting $D_\nu=\sum_{k=1}^n D_\nu^{k}$. Similarly, the constraints $\|\nabla \mu_\xi\|_\infty\le D_\xi^\mu$ in Optimization Problem \ref{SDP2} can be implemented as%
\begin{equation*}
  -D_\xi^\mu\le [\nabla \mu_\xi]_k \le D_\xi^\mu\ \ \text{for $k=1:n$}.%
\end{equation*}
\end{remark}

\begin{remark}
The constraints above are easily transferred into the standard form $\sum_{i=1}^m F_i y_i-F_0\succeq 0$, $F_0,F_1,\ldots,F_m\in\R^{n\times n}$ constant matrices and $y_1,y_2,\ldots,y_m\in\R$ the variables, for semidefinite programming (SDP) with linear matrix inequality (LMI) constraints. See, e.g., \cite[Rem.~4.10]{GH1} for a similar transfer.
\end{remark}

\begin{remark}
The Optimization Problem \ref{SDP1} always has a feasible solution if the parameter $\mu$ is chosen large enough.  Indeed, even for a fixed $P\succeq \epsilon_0 I$ the constraints are fulfilled for a large enough $\mu$.
\end{remark}
% XXX Main comment 1

\subsection{Feasible solution to Optimization Problem \ref{SDP1}}

A feasible solution of the Optimization Problem \ref{SDP1} returns a matrix $P(x_k) = \left(P_{ij}(x_k)\right)_{i,j=1:n}$ at each vertex $x_k$ of the triangulation $\cT$ and values $C_\nu$ and $D_\nu$ for each simplex $\T_\nu\in\cT$. From these we can easily obtain $A(x_k)$ and $E_\nu$ as in Constraints 4 at each vertex $x_k$ and for each simplex $\T_\nu$, respectively.%

We define the CPA (\emph{continuous piecewise affine}) metric $M$ by affine interpolation on each simplex.%

\begin{definition}[CPA interpolation]\label{CPA-def}
Let $\cT$ be a triangulation in $\R^n$ with $\cD_\cT = \bigcup_{\T_\nu\in \cT}\T_\nu$. Let $P_{ij}(x_k)$ be fixed by a feasible solution to the Optimization Problem \ref{SDP1}. An $x \in \T_\nu  = \co(x_0,\ldots,x_n)\in\cT$ can be written uniquely as $x=\sum_{k=0}^{n}\lambda_k x_k$ with $\lambda_k\in [0,1]$ and $\sum_{k=0}^{n}\lambda_k=1$ and we define%
\begin{equation*}
  P_{ij}(x) := \sum_{k=0}^{n}\lambda_k P_{ij}(x_k)%
\end{equation*}
and%
\begin{equation}\label{eq_pmatrix}
  P(x) := \begin{pmatrix}
       P_{11}(x) & P_{12}(x)& \cdots & P_{1n}(x) \\
       P_{21}(x) & P_{22}(x)& \cdots & P_{2n}(x) \\
       \vdots & \vdots & \ddots & \vdots \\
       P_{n1}(x) & P_{n2}(x) & \cdots & P_{nn}(x) \\
     \end{pmatrix}.
\end{equation}
We refer to the functions $P_{ij}$ and $P$ as the CPA interpolations of the values $P_{ij}(x_k)$ and $P(x_k)$, respectively, where the $x_k$ are the vertices of the simplices in $\cT$. Furthermore, we write $\CPA[\cT]$ for the space of all piecewise affine functions on $\cD_\cT$ defined in this way by interpolation of the values on the vertices of $\cT$.%
\end{definition}

The following lemma can be proved exactly as \cite[Lem.~4.13]{GH1}.%

\begin{lemma}\label{le1}
The matrix $P(x)$ in \eqref{eq_pmatrix} is symmetric and positive definite for all $x\in\cD_\cT$.
\end{lemma}

\begin{definition}[Orbital derivative]\label{basic}
Let $P(x)$ be as in Definition \ref{CPA-def} and fix a point $x\in\cD^\circ_\cT$.  As shown in the proof of \cite[Lem.~4.7]{GH1}, there exists a $\T_\nu=\co(x_0,\ldots,x_n)\in \cT$ and a number $\theta^*>0$ such that $x+\theta f(x)\in \T_\nu$ for all $\theta\in [0,\theta^*]$. Then $x=\sum_{k=0}^{n}\lambda_k x_k$ with $\lambda_k\in [0,1]$, $\sum_{k=0}^{n}\lambda_k=1$, and we define the orbital derivative $\dot{P}_{ij}(x)$ of $P_{ij}$ at $x$ as%
\begin{equation*}
  \dot{P}_{ij}(x) := w_{ij}^\nu \cdot f(x).%
\end{equation*}
\end{definition}

Our definition of the orbital derivative is natural, because with $t\mapsto \phi(t,x)$ as the solution to $\dot x=f(x)$ crossing $x$ at time $t=0$ and for any locally Lipschitz function $g:\R^n\to\R$, we have (cf.~\cite[Thm.~1.17]{Mar})%
\begin{align*}
  \limsup_{h\to 0+}\frac{g(\phi(h,x))-g(x)}{h}=\limsup_{h\to 0+}\frac{g(x+hf(x))-g(x)}{h}%
\end{align*}
and with $\T_\nu$ chosen for $x$ as in Definition \ref{basic}, we have%
\begin{equation*}
  \limsup_{h\to 0+}\frac{P_{ij}(x+hf(x))-P_{ij}(x)}{h} = w_{ij}^\nu \cdot f(x).%
\end{equation*}
Note that $P_{ij}\big|_{\T_\nu}$ is an affine function and its gradient $w^{\nu}_{ij}$ was defined in Constraints 3 of Optimization Problem \ref{SDP1} and is the same vector for all points $x\in \T_\nu$.%

Before proceeding to prove the implications of Optimization Problem \ref{SDP1}, let us first recall a few elementary relations about matrix norms. For an $A\in\R^{n\times n}$ we define%
\begin{equation*}
  \|A\|_{\max}:=\max_{i,j=1:n}|a_{ij}|\quad \text{and}\quad \|A\|_p:=\max_{\|x\|_p=1}\|A x\|_p\  \ \text{for $p=1,2,\infty$.}%
\end{equation*}
The following relations hold:%
\begin{equation*}
  \|A\|_{\max}\le \|A\|_2 \le n \|A\|_{\max},\ \|A\|_2\le \sqrt{n}\|A\|_1,\ \text{and}\ \|A\|_2\le \sqrt{n}\|A\|_\infty.%
\end{equation*}
For a symmetric and positive definite $A$, the largest singular value $\lambda_{\max}$ of $A$, which equals $\|A\|_2$ and is the largest of its eigenvalues, is the smallest number such that $A \preceq \lambda_{\max}I$. Further,%
\begin{equation*}
  \|A\|_1 = \max_{j=1:n}\sum_{i=1}^n|a_{ij}| = \|A\trn\|_\infty.%
\end{equation*}

We will now relate $\dot{P}(x)$ to $\dot{P}(x_k)$, as well as $P(x)\rmD f(x)$ to $P(x_k)\rmD f(x_k)$. For the proof we will need the following auxiliary result, see \cite[Prop.~4.1 and Cor.~4.3]{BGH}. The notation is as in Optimization Problem \ref{SDP1}.%

\begin{lemma}\label{help2}
Denoting the Hessian of $f$ by $H(x):=\left(\frac{\partial^2 f(x)}{\partial x_i\partial x_j}\right)_{ij}$, we have%
\begin{equation*}
  \left\|f(x)-\sum_{k=0}^{n}\lambda_k f(x_k)\right\|_\infty\le\max_{x\in \T_\nu}\|H(x)\|_2 h_\nu^2 \le n B_\nu h_\nu^2.%
\end{equation*}
\end{lemma}

%where the last inequality follows from $\|H(x)\|_2\le %n\|H(x)\|_{\max}$.

\begin{lemma}\label{4.12}
Consider a feasible solution to Optimization Problem \ref{SDP1} and let $P$ be defined as in Definition \ref{CPA-def}. Fix a point $x\in\cD^\circ_\cT$ and a corresponding simplex $\T_\nu=\co(x_0,x_1,\ldots,x_n)\in \cT$ as in Definition \ref{basic}. Set%
\begin{equation*}
  A(y) := P(y)\rmD f(y) + \rmD f(y)\trn P(y) + (w_{ij}^\nu \cdot f(y))_{i,j=1:n}%
\end{equation*}
for all $y\in \T_\nu$. Then we have the following estimate{\rm :}%
\begin{equation}\label{Aest}
  \left\|A(x)-\sum_{k=0}^n \lambda_k A(x_k)\right\|_2 \le h_\nu^2 E_\nu.%
\end{equation}
\end{lemma}

\begin{proof}
We show this in several steps:%

{\bf Step 1: Entry-wise bounds on $\mathbf{\dot{P}(x)}$}\\
The estimate
\begin{equation}
\label{Mdbound}
\left|w^\nu_{ij}\cdot f(x)-\sum_{k=0}^{n}\lambda_k w^\nu_{ij}\cdot f(x_k)\right|
\le n B_\nu D_\nu h_\nu^2
\end{equation}
follows by H\"older's inequality, Constraints 3, and Lemma \ref{help2}:%
\begin{equation*}
  \left|w^\nu_{ij}\cdot \left(f(x)
-\sum_{k=0}^{n}\lambda_k  f(x_k)\right)\right|\le
\|w^{\nu}_{ij}\|_1\left\| f(x)-\sum_{k=0}^{n}\lambda_k f(x_k)\right\|_\infty \le D_\nu nB_\nu h_\nu^2.
\end{equation*}

{\bf Step 2: Entry-wise bounds on $\mathbf{P(x) \rmD f(x)}$ and $\mathbf{\rmD f(x)\trn P(x)}$}\\
We show that%
\begin{equation}
\label{MDbound}
\left|[P(x)\rmD f(x)]_{ij} - \sum_{k=0}^{n}\lambda_k [P(x_k)\rmD f(x_k)]_{ij}\right|\le
nh_\nu^2(2\sqrt{n}B_\nu D_\nu +nB_{3,\nu}C_\nu).%
\end{equation}
%The estimate
%\begin{equation}
%\label{MDbound2}
%\left|(D f(x)^TM(x))_{ij}-\sum_{k=0}^{n}\lambda_k (D f(x_k)^TM(x_k))_{ij}\right|\le nh_\nu^2(2\sqrt{n} B_\nu D_\nu+nB_{3,\nu}C_\nu)
%\end{equation}
%follows analogously.

Consider two scalar-valued functions $g,h\in C^2(\T_\nu)$. We apply Lemma \ref{help2} to $gh$, yielding%
\begin{equation}
  \left|g(x)h(x)-\sum_{k=0}^{n}\lambda_k g(x_k)h(x_k)\right| \le \max_{y\in \T_\nu}\|H(y)\|_2h_\nu^2,\label{help3}
\end{equation}
where the matrix $H(y)$ is defined by $[H(y)]_{rs}:=\frac{\partial^2 (gh)(y)}{\partial x_r\partial x_s}$. Set $g(y) := P_{il}(y)$.  Since $P_{il}(y) = w^\nu_{il}\cdot(y-x_0)+P_{il}(x_0)$, we get $\frac{\partial g}{\partial x_s}(y)=[w_{il}^\nu]_s$ and
$\frac{\partial^2 g}{\partial x_r\partial x_s} (y)=0$ for all $y\in \T_\nu$. Hence,%
\begin{equation*}
  \frac{\partial }{\partial x_s}gh=\frac{\partial g}{\partial x_s} h+g
\frac{\partial h}{\partial x_s} = [w_{il}^\nu]_s h +g
\frac{\partial h}{\partial x_s}
\end{equation*}
and then%
\begin{equation*}
\frac{\partial^2 }{\partial x_r\partial x_s}gh = [w_{il}^\nu]_s\frac{\partial h}{\partial x_r}+\frac{\partial g}{\partial x_r}\frac{\partial h}{\partial x_s}+g\frac{\partial^2 h}{\partial x_r\partial x_s}
= [w_{il}^\nu]_{s}\frac{\partial h}{\partial x_r}+[w_{il}^\nu]_r\frac{\partial h}{\partial x_s}+P_{il}\frac{\partial^2 h}{\partial x_r\partial x_s}.%
\end{equation*}
Now set $h(y) := [\rmD f(y)]_{lj}$. Then $\frac{\partial h}{\partial x_r}=\frac{\partial^2 f_l}{\partial x_r \partial x_j}$ and
$\frac{\partial^2 h}{\partial x_r\partial x_s}=\frac{\partial^3 f_l}{\partial x_r\partial x_s\partial x_j}$ and thus%
\begin{equation*}
  \left|[H(y)]_{rs}\right| = \left|\frac{\partial^2 (gh)(y)}{\partial x_r\partial x_s}\right|
\le |[w^\nu_{il}]_s|B_\nu+|[w^\nu_{il}]_r|B_\nu + |P_{il}(y)|B_{3,\nu}.%
\end{equation*}
Using in succession for any  $H_1,H_2,H_3\in \R^{n\times n}$ that%
\begin{equation*}
  \|H_1+H_2+H_3\|_2\le \|H_1\|_2+\|H_2\|_2+\|H_3\|_2%
\end{equation*}
and%
\begin{equation*}
  \|H_1\|_2\le \sqrt{n} \|H_1\|_\infty,\  \|H_2\|_2\le \sqrt{n} \|H_2\|_1,\ \text{and}\ \|H_3\|_2\le n \|H_3\|_{\max},%
\end{equation*}
this delivers%
\begin{align}
\|H(y)\|_2\
&\le \sqrt{n}\|w^\nu_{il}\|_1B_\nu+\sqrt{n}\|w^\nu_{il}\|_1B_\nu+nB_{3,\nu}\max_{x\in \T_\nu}\max_{1\le i\le l\le n} |P_{il}(x)|\n\\
&\le 2\sqrt{n}B_\nu D_\nu +nB_{3,\nu}C_\nu, \label{HestX}
\end{align}
because we have $|P_{il}(y)| \le \|P(y)\|_2 \le C_\nu$ by Constraints 2.%

Hence, \eqref{help3} and \eqref{HestX} establish%
\begin{align*}
&\left|
[P(x)\rmD f(x)]_{ij}-\sum_{k=0}^{n}\lambda_k [P(x_k)\rmD f(x_k)]_{ij}\right|
= \left|\sum_{l=1}^n P_{il}(x)[\rmD f(x)]_{lj}-\sum_{l=1}^n\sum_{k=0}^{n}\lambda_k P_{il}(x_k)[\rmD f(x_k)]_{lj}\right|\n \\
 &\le \sum_{l=1}^n\left| P_{il}(x)[\rmD f(x)]_{lj}-\sum_{k=0}^{n}\lambda_k P_{il}(x_k)[\rmD f(x_k)]_{lj}\right| \le
n \cdot (2\sqrt{n}B_\nu D_\nu +nB_{3,\nu}C_\nu)\cdot h_\nu^2.%
\end{align*}

%\begin{align*}
%&\left|
%[P(x)\rmD f(x)]_{ij}-\sum_{k=0}^{n}\lambda_k [P(x_k)\rmD f(x_k)]_{ij}\right| \\
%&\ \ \ =\left|\sum_{l=1}^n P_{il}(x)[\rmD f(x)]_{lj}-\sum_{l=1}^n\sum_{k=0}^{n}\lambda_k P_{il}(x_k)[\rmD f(x_k)]_{lj}\right|\n \\
%&\ \ \ \le \sum_{l=1}^n\left| P_{il}(x)[\rmD f(x)]_{lj}-\sum_{k=0}^{n}\lambda_k P_{il}(x_k)[\rmD f(x_k)]_{lj}\right|\n \\
%&\ \ \ \le
%n \cdot (2\sqrt{n}B_\nu D_\nu +nB_{3,\nu}C_\nu)\cdot h_\nu^2.%
%\end{align*}

{\bf Step 3: Bounds on matrices}\\
From the definition of $A(y)$ we get
\begin{align*}
  \left\|A(x)-\sum_{k=0}^n \lambda_k A(x_k)\right\|_2 &\le \left\|P(x)\rmD f(x) - \sum_{k=0}^n \lambda_k P(x_k)\rmD f(x_k) \right\|_2 \\
&+\left\|\rmD f(x)\trn P(x) - \sum_{k=0}^n \lambda_k \rmD f(x_k)\trn P(x_k) \right\|_2 \\
&+\left\| (w_{ij}^\nu \cdot f(x))_{i,j=1:n} - \sum_{k=0}^n \lambda_k (w_{ij}^\nu \cdot f(x_k))_{i,j=1:n} \right\|_2.%
\end{align*}
The first two norms on the right-hand side are equal because $P$ is symmetric and therefore the matrices in the norms are conjugate.  The entry-wise bounds \eqref{Mdbound} and \eqref{MDbound} together with $\|H\|_2 \le n\|H\|_{\max}$ for any $H\in\R^{n\times n}$ now deliver \eqref{Aest}.
\end{proof}

Using Lemma \ref{4.12} we can now establish, that the parameter $\mu$ in Optimization Problem \ref{SDP1} is an upper bound on the generalized eigenvalues of the matrix pair $(A(x),P(x))$ for all $x\in\cD_\cT^\circ$. For completeness, we first give a short description of generalized eigenvalues as needed here.%

The generalized eigenvalue problem for two symmetric matrices $A,B\in\R^{n\times n}$, $B \succ 0$, is to find values $\lambda_i \in \R$ and corresponding nonzero vectors $x_i\in\R^n$ such that $Ax_i=\lambda_iBx_i$ for $i=1:n$. Since $B\succ 0$, we can define $B^{\frac{1}{2}} := O\trn D^{\frac{1}{2}} O$, where $B=O\trn D O$ is the spectral decomposition of $B$, i.e.\ $O$ is orthogonal and $D$ is a diagonal matrix with strictly positive entries on the diagonal. $D^{\frac{1}{2}}$ is then canonically defined as the diagonal matrix with the square-roots of the entries of $D$ on the diagonal. Further, $B^{-\frac{1}{2}}:=(B^{\frac{1}{2}})^{-1}$. The matrix $C = B^{-\frac{1}{2}}AB^{-\frac{1}{2}}$ is then symmetric and if $\lambda\in\R$ is an eigenvalue of $C$ with corresponding eigenvector $y\in\R^n$, then $B^{-\frac{1}{2}}AB^{-\frac{1}{2}}y=\lambda y$. From this $AB^{-\frac{1}{2}}y=\lambda B^{\frac{1}{2}}y=B B^{-\frac{1}{2}}y$ or $Ax=\lambda B x$ with $x=B^{-\frac{1}{2}}y$, i.e.\ $\lambda$ is a generalized eigenvalue for the matrix pair $A$ and $B$ and $x$ is a corresponding generalized eigenvector. With $\{y_i\}$ as an orthonormal set of eigenvectors of $C$, the generalized eigenvectors $x_i = B^{-\frac{1}{2}}y_i$ are thus a basis of $\R^n$ and $x_i\trn Bx_j=y_i\trn B^{-\frac{1}{2}}BB^{-\frac{1}{2}}y_j = y_i\trn y_j=\delta_{ij}$. It follows from an easy calculation that if $\lambda_{\max}$ is the largest generalized eigenvalue of the pair $(A,B)$, then%
\begin{equation*}
  0\succeq A-\mu B\ \ \text{if and only if}\ \ \mu \ge \lambda_{\max}.%
\end{equation*}
Also note that $0 \succeq A-\mu B + \alpha I$ for an $\alpha \ge 0$ clearly implies $0 \succeq A-\mu B$. Since $B\succ 0$, the smallest eigenvalue of $B$ is given by $\|B^{-1}\|_2^{-1}$ and then $\|B^{-1}\|_2^{-1} I \preceq B$ and with $\alpha \ge0$ we get%
\begin{equation}\label{LOKAX}
  A-\mu B+\alpha I \preceq A-\left(\mu - \alpha\|B^{-1}\|_2\right)B\preceq 0 \ \ \text{if}\ \ \mu - \alpha\|B^{-1}\|_2 \ge \lambda_{\max}.%
\end{equation}

From this discussion on generalized eigenvalues and Lemma \ref{4.12} we can draw the following conclusion:%

\begin{theorem}\label{thm_amr1}
Assume that we have a feasible solution to Optimization Problem \ref{SDP1} with parameter $\mu \ge 0$ and let $P(x)$ be defined from the feasible solution as in Definition \ref{CPA-def} and define for every $x\in\cD_\cT^\circ$ the matrix%
\begin{equation*}
  A(x) := P(x)\rmD f(x) + \rmD f(x)\trn P(x) + (w_{ij}^\nu \cdot f(x))_{i,j=1:n}.%
\end{equation*}
Denote for every $x\in\cD_\cT^\circ$ by $\lambda_{\max}(x)$ the largest generalized eigenvalue of the matrix pair $(A(x),P(x))$. Then $\mu \ge \lambda_{\max}(x)$ for every $x\in\cD_\cT^\circ$.
\end{theorem}

\begin{proof}
Let $x\in\cD_\cT^\circ$ be arbitrary and fix an $\T_\nu\in\cT$ as in Definition \ref{basic} such that $x\in\T_\nu$. By Lemma \ref{4.12} we have%
\begin{equation*}
  \left\|A(x)-\sum_{k=0}^n \lambda_k A(x_k) \right\|_2 \le h_\nu^2E_\nu,%
\end{equation*}
from which%
\begin{equation*}
  A(x) - \sum_{k=0}^n \lambda_kA(x_k) \preceq \left\|A(x)-\sum_{k=0}^n \lambda_k A(x_k) \right\|_2 I \preceq h_\nu^2 E_\nu I%
\end{equation*}
follows. But then%
\begin{align*}
A(x)-\mu P(x) &\preceq \left\|A(x)-\sum_{k=0}^n \lambda_k A(x) \right\|_2 I + \sum_{k=0}^n \lambda_k A(x_k) - \mu \sum_{k=0}^n\lambda_k P(x_k) \\
& \preceq \sum_{k=0}^n\lambda_k\left[A(x_k) - \mu P(x_k) + h_\nu^2E_\nu I\right] \preceq 0%
\end{align*}
by Constraints 4. The assertion of the lemma now follows by the discussion on generalized eigenvalues above.
\end{proof}

After we have found a parameter $\mu\ge0$ and a triangulation $\cT$ such that Optimization Problem \ref{SDP1} has a feasible solution, we can use this solution as the input to another optimization problem to get bounds on $h_{\tp}(\phi;K)$ as in Theorem \ref{thm_matpog}. We now use $P(x)$ computed by Optimization Problem \ref{SDP1} to compute two functions, $\mu,V:\cD_\cT \to \R$. The function $\mu$, satisfying $0\le \lambda_{\max}(x) \le \mu(x)\le \mu$, is a local upper bound on the largest generalized eigenvalue $\lambda_{\max}(x)$ of $(A(x),P(x))$ and $V$ is a Lyapunov-type function as used in Theorem \ref{thm_matpog}.%

\begin{opt}\label{SDP2}
Given is a feasible solution to Optimization Problem \ref{SDP1} for a system $\dot x=f(x)$, $f\in C^3(\R^n;\R^n)$, with a triangulation $\cT$, a parameter $\mu\ge 0$, and an upper bound $\widetilde m$ on the number of positive generalized eigenvalues of the matrix pairs $(A(x),P(x))$ from the feasible solution. Further, a refined triangulation $\cT^*$ of $\cT$ is given, i.e.\ $\cT^*$ is a triangulation as in Definition \ref{scdef}, $\cD_{\cT^*}=\cD_\cT$, and each simplex $\T_\nu\in\cT$ is the union of simplices $\T_\xi$ in $\cT^*$.%

The optimization problem is a semidefinite problem with linear matrix inequality constraints.%

{\bf Constants}
The constants used in problem are\,{\rm :}%
\begin{enumerate}
\item The diameter $h_\xi$ of each simplex $\T_\xi\in \cT^*$\,{\rm :}%
\begin{equation*}
  h_\xi :=\diam (\T_\xi)=\max_{x,y\in \T_\xi}\|x-y\|_2
\end{equation*}
\item $D_\nu$ and $E_\nu$ for each simplex $\T_\nu\in\cT$ -- delivered by the feasible solution to Optimization Problem \ref{SDP1}%
\item Upper bounds $B^*_\xi$ on the second-order derivatives of the components of $f$ on each simplex $\T_\xi \in \cT^*$, just as in \eqref{Bdef} but for the simplices in $\cT^*$. For example, one can set $B^*_\xi := B_\nu$ for every $\T_\xi\in \cT^*$ fulfilling $\T_\xi\subset \T_\nu\in\cT$.%
\end{enumerate}
We additionally use the functions $A,P:\cD_\cT^*\to \R^{n\times n}$ from the feasible solution to Optimization Problem \ref{SDP1}.%

{\bf Variables}
The variables of the semidefinite feasibility problem are\,{\rm :}%
\begin{enumerate}
\item $\mu(x_k)\in \mathbb R$ for all vertices $x_k$ of all simplices $\T_\nu=\co(x_0,\ldots,x_{n})\in \cT^*$ -- upper bound on the largest generalized eigenvalue%
\item $D_\xi^\mu \in \mathbb R_0^+$ for all simplices $\T_\xi \in \cT^*$ -- upper bound on the gradient of $\mu$%
\item $V(x_k)\in \mathbb R$ for all vertices $x_k$ of all simplices $\T_\nu=\co(x_0,\ldots,x_{n})\in \cT^*$ -- value of the Lyapunov-type function at $x_k$%
\item $D_\xi^V\in \mathbb R$ for all simplices $\T_\xi \in \cT^*$ -- upper bound on the gradient of $V$%
\item $Q\in \mathbb R_0^+$ -- the quantity to be minimized%
\end{enumerate}

{\bf Objective}
\begin{equation*}
\text{minimize}\ Q
\end{equation*}
% Note that Constraints 1, 2, and 4 are LMIs and Constraint 2 is a numerical linear inequality.  Further, since each vertex is common to more simplices the constraints are coupled.

{\bf Constraints}

\begin{enumerate}
\item {\bf Bound on the gradient of $\mathbf{\mu}$}\\
For each simplex $\T_\xi=\co(x_0,\ldots,x_n)\in\cT^*$\,{\rm :}
\begin{equation*}
  \|\nabla \mu_\xi\|_\infty \le D^\mu_\xi%
\end{equation*}
See Remark \ref{nugradrem} for details.
\item {\bf $\mathbf{\mu(x)}$ an upper bound on the generalized eigenvalues}\\
For each simplex $\T_\xi=\co(x_0,\ldots,x_n)\in\cT^*$ and each vertex $x_k$ of $\T_\xi$\,{\rm :}
\begin{equation}\label{LMIaux}
  A(x_k)-\mu(x_k) P(x_k)+h_\xi^2 (E_\nu + 2n\sqrt{n} D_\nu D_\xi^\mu) I \preceq 0%
\end{equation}
$E_\nu$ and $D_\nu$ correspond to the simplex $\T_\nu\in\cT$ such that $\T_\xi\subset \T_\nu$.
\item {\bf Bound on the gradient of $\mathbf{V}$}\\
For each simplex $\T_\xi=\co(x_0,\ldots,x_n)\in\cT^*$\,{\rm :}%
\begin{equation*}
\|\nabla V_\xi\|_1\le D^V_\xi
\end{equation*}
Here $\nabla V_\xi:=\nabla  V\big|_{\T_\xi}(x)$ for all $x\in \T_\xi$. It is constructed exactly as the vectors $w^\nu_{ij}$ in Optimization Problem \ref{SDP1}, see also Remark \ref{nugradrem}.
\item {\bf Upper bound on the sum of generalized eigenvalues}\\
For each simplex $\T_\xi=\co(x_0,\ldots,x_n)\in\cT^*$ and each vertex $x_k$ of $\T_\xi$\,{\rm :}
\begin{equation}\label{LPaux}
  \nabla V_\xi \cdot f(x_k) + h_\xi^2\cdot n B^*_\xi D_\xi^V  + \widetilde m\mu(x_k) \le Q%
\end{equation}
\end{enumerate}
\end{opt}
The orbital derivative of $V$ is defined as in Definition \ref{basic}, but with the triangulation $\cT^*$ of course.%

\begin{remark}
The Optimization Problem \ref{SDP2} clearly has a solution $Q \geq \tilde{m}\mu$, where $\mu$ is the parameter from Optimization Problem \ref{SDP1} chosen such that it has a feasible solution.  Indeed, just set $\mu(x_k)=\mu$ and $V(x_k)=0$ for all vertices $x_k$ of $\cT^*$.
\end{remark}

% XXX Main comment 1

\begin{theorem}\label{X4.12}
Consider a feasible solution to Optimization Problem \ref{SDP2} and let $\mu(x)$ and $V(x)$ be constructed as in Definition \ref{CPA-def}. Then for every $x\in\cD_{\cT^*}^\circ$ we have%
\begin{equation*}
  A(x) - \mu(x) P(x) \preceq 0 \mbox{\quad and \quad} V'(x) + \widetilde m \mu(x) \le Q.%
\end{equation*}
\end{theorem}

\begin{proof}
Fix a point $x\in\cD^\circ_{\cT^*}$ and a corresponding simplex $\T_\xi=\co(x_0,x_1,\ldots,x_n)\in \cT^*$ as in Definition \ref{basic}. Further, denote by $\T_\nu$ the simplex in $\cT$ such that $\T_\xi\subset \T_\nu$. Now%
\begin{align}
&\left\|A(x)-\mu(x)P(x) - \sum_{k=0}^n\lambda_k\left[A(x_k)-\mu(x_k)P(x_k)\right]\right\|_2\n\\
&\ \ \ \le \left\|A(x)- \sum_{k=0}^n\lambda_kA(x_k)\right\|_2+\left\|\mu(x)P(x)-\sum_{k=0}^n\lambda_k\mu(x_k)P(x_k)\right\|_2. \label{XAXA}
\end{align}
Just as in the proof of Lemma \ref{4.12}, we can show that%
\begin{equation}\label{OOTL}
  \left\|A(x)- \sum_{k=0}^n\lambda_kA(x_k)\right\|_2 \le h_\xi^2 E_\nu.%
\end{equation}
For the second norm on the right-hand side of \eqref{XAXA} consider two scalar-valued functions $g,h:\T_\xi\to \R$, where $g,h\in C^2$.
 We apply Lemma \ref{help2} to $gh$, yielding%
\begin{equation*}
  \left|g(x)h(x)-\sum_{k=0}^{n}\lambda_k g(x_k)h(x_k)\right| \le \max_{y\in \T_\xi}\|H(y)\|_2h_\xi^2,%
\end{equation*}
where the matrix $H(y)$ is defined by $[H(y)]_{rs}:=\frac{\partial^2 (gh)(y)}{\partial x_r\partial x_s}$. Set $g(y):= P_{il}(y)$ and $h(y) := \mu(y)$. Since $\T_\xi\subset \T_\nu$, we have $P_{il}(y)=w^\nu_{il}\cdot(y-x_0)+P_{il}(x_0)$ and because $\mu(y)$ is defined as a CPA interpolation, we have $\mu(y)=[\nabla \mu_\xi]\cdot(y-x_0)+\mu(x_0)$. Thus, we have $\frac{\partial g}{\partial x_s}(y)=[w_{il}^\nu]_s$, $\frac{\partial^2 g}{\partial x_r\partial x_s}(y)=0$, $\frac{\partial h}{\partial x_s}(y)=[\nabla \mu_\xi]_s$, and
$\frac{\partial^2 h}{\partial x_r\partial x_s} (y)=0$ for all $y\in \T_\nu$. Hence,%
\begin{equation*}
  \frac{\partial}{\partial x_s}gh = \frac{\partial g}{\partial x_s} h+g
  \frac{\partial h}{\partial x_s} = [w_{il}^\nu]_s h +g[\nabla \mu_\xi]_s%
\end{equation*}
and%
\begin{align*}
\frac{\partial^2 }{\partial x_r\partial x_s}gh&=[w_{il}^\nu]_s\frac{\partial h}{\partial x_r}+\frac{\partial g}{\partial x_r} [\nabla \mu_\xi]_s =[w_{il}^\nu]_s[\nabla \mu_\xi]_r +[w_{il}^\nu]_r[\nabla \mu_\xi]_s.%
\end{align*}
Thus%
\begin{equation*}
  \left|[H(y)]_{rs}\right| = \left|[w_{il}^\nu]_s[\nabla \mu_\xi]_r +[w_{il}^\nu]_r[\nabla \mu_\xi]_s \right|
\le \left|[w_{il}^\nu]_s[\nabla \mu_\xi]_r\right| + \left|[w_{il}^\nu]_r[\nabla \mu_\xi]_s\right|.%
\end{equation*}

Using that for any $H_1,H_2\in \R^{n\times n}$ we have%
\begin{equation*}
  \|H_1+H_2\|_2\le \|H_1\|_2+\|H_2\|_2 \le \sqrt{n} \|H_1\|_1 + \sqrt{n} \|H_2\|_\infty,%
\end{equation*}
we get%
\begin{equation*}
  \|H(y)\|_2 \le 2\sqrt{n} \|w_{il}^\nu\|_1 \|\nabla \mu_\xi\|_\infty \le 2\sqrt{n} D_\nu D_\xi^\mu,%
\end{equation*}
because $\|w_{il}^\nu\|_1 \le D_\nu$ by Constraints 3 in Optimization Problem \ref{SDP1} and $\|\nabla \mu_\xi\|_\infty \le D_\xi^\mu $ by Constraints 1 in Optimization Problem \ref{SDP2}.%

We have shown that%
\begin{equation*}
  \left\|\mu(x)P(x)-\sum_{k=0}^n\lambda_k\mu(x_k)P(x_k)\right\|_{\max} \le h_{\xi}^2 \cdot2\sqrt{n} D_\nu D_\xi^\mu%
\end{equation*}
and it follows that%
\begin{equation}\label{OOTL2}
  \left\|\mu(x)P(x) - \sum_{k=0}^n\lambda_k\mu(x_k)P(x_k)\right\|_2 \le h_{\xi}^2 \cdot 2n\sqrt{n} D_\nu D_\xi^\mu.%
\end{equation}
Thus, we have by \eqref{XAXA}, \eqref{OOTL}, \eqref{OOTL2}, and Constraints 2 of Optimization Problem \ref{SDP2}%
\begin{align*}
  A(x)-\mu(x)P(x) &\preceq \sum_{k=0}^n\lambda_k[A(x_k)-\mu(x_k)P(x_k)] + \left\|A(x)- \sum_{k=0}^n\lambda_kA(x_k)\right\|_2I\\
&\ \ \ + \left\|\mu(x)P(x)-\sum_{k=0}^n\lambda_k\mu(x_k)P(x_k)\right\|_2I \\
& \preceq \sum_{k=0}^n\lambda_k[A(x_k)- \mu(x_k)P(x_k) + h_\xi^2(E_\nu+2n\sqrt{n}D_\nu D_\xi^\mu)I] \preceq 0%
\end{align*}
and the estimate \eqref{LMIauxX} follows and, as before, it also follows that $\mu(x)$ is an upper bound on the largest generalized eigenvalue of the matrix pair $(A(x),P(x))$ for every $x\in\cD_\cT^\circ$.%

Let $x\in\cD_{\cT^*}$ and $\T_\xi$ be as above. We now show the implications of Constraints 4. By H\"older's inequality and Lemma \ref{help2} we get%
\begin{align*}
& \nabla V_\xi\cdot f(x) \le \sum_{k=0}^n \lambda_k \nabla V_\xi \cdot f(x_k) + \left|\nabla V_\xi\cdot f(x) -\sum_{k=0}^n \lambda_k \nabla V_\xi \cdot f(x_k)\right| \\
& \le \sum_{k=0}^n \lambda_k \nabla V_\xi \cdot f(x_k) + \|\nabla V_\xi\|_1 \left\|f(x)-\sum_{k=1}^n \lambda_k f(x_k)\right\|_\infty
 \le \sum_{k=0}^n \lambda_k \nabla V_\xi \cdot f(x_k) + h_\xi^2 \cdot n B^*_\xi D_\xi^V%
\end{align*}
and then by Constraints 4 of Optimization Problem \ref{SDP2}%
\begin{equation*}
  \nabla V_\xi\cdot f(x) + \widetilde m \mu(x) \le \sum_{k=0}^n \lambda_k \left[\nabla V_\xi \cdot f(x_k) + h_\xi^2 \cdot n B^*_\xi D_  \xi^V + \widetilde m \mu(x_k)\right] \le \sum_{k=0}^n \lambda_k Q = Q,%
\end{equation*}
i.e.\ the estimate \eqref{mainQ}.
\end{proof}

The following lemma shows that the piecewise affine functions $P$, $V$ and $\mu$ computed by our algorithm can be approximated by smooth functions asymptotically satisfying the same inequalities. Hence, we do not get into trouble because of the differentiability assumptions in Theorem \ref{thm_matpog}.%

For any subset $\cD\subset \R^n$ and $\varepsilon>0$ define $\cD_{-\varepsilon}:= \{x\in\cD\,:\ B_\varepsilon(x)\subset \cD\}$. Define $\phi:\R^n\to \R_+$, $\phi(x) := C\exp(-1/(1-\|x\|_2))$ for $\|x\|_2<1$ and $\phi(x) := 0$ otherwise and choose the constant $C$ such that $\int_{\R^n} \phi(y)\, \rmd y = 1$. For an $\varepsilon>0$ define%
\begin{equation*}
  \widetilde\phi_\varepsilon(x) := \frac{\phi(x/\varepsilon)}{\varepsilon^n}.%
\end{equation*}
For a locally integrable $g:\cD\to\R$, $\cD\subset \R^n$, define the function $g_\varepsilon := g*\widetilde\phi_\varepsilon$, i.e.\ $g_\varepsilon(x) = \int_{\cD} g(y) \widetilde{\phi}_\varepsilon(y-x) \rmd y$. It is well-known that $g_\varepsilon,\widetilde\phi_\varepsilon \in C^\infty(\R^n)$ and if $g$ is continuous on $\cD\subset\R^n$ and $\cK\subset \cD^\circ$ is compact, then the functions $g_\varepsilon$ approximate $g$ uniformly on $\cK$, i.e.\ $\max_{x\in \cK}|g_\varepsilon(x)-g(x)|\to 0$ as $\varepsilon \to 0+$.%

\begin{lemma}\label{cpasmoothlemma}
Assume $g\in \CPA[\cT]\to\R$ (cf.~Definition \ref{CPA-def}), $\varepsilon>0$, and denote by $w_\nu$ the gradient of $g$ on $\T_\nu$. That is, $g(x) = w_\nu\cdot x+b_\nu$ on $\T_\nu$. Then for every $x\in(\cD_\cT)_{-\varepsilon}$ we have%
\begin{equation*}
  \nabla g_\varepsilon(x) = \sum_{\nu} \alpha_{\nu}^{x,\varepsilon} w_\nu,\ \ \text{where}\ \ \alpha_{\nu}^{x,\varepsilon} :=  \int_{\T_\nu\cap B_\varepsilon(x)}\widetilde{\phi}_{\ep}(x - y)\, \rmd y.%
\end{equation*}
Especially, the nonnegative numbers $\alpha_{\nu}^{x,\varepsilon}$ only depend on $x$ and $\varepsilon>0$ and not on the function $g$ and they sum to one.
\end{lemma}

\begin{proof}
This follows from the following calculation, using integration by parts:%
\begin{align*}
  \nabla g_{\ep}(x) &= \int \nabla_x\widetilde{\phi}_{\ep}(x - y) g(y)\, \rmd y = \sum_{\nu} \int_{\T_\nu\cap B_\varepsilon(x)} -\nabla_y\widetilde{\phi}_{\ep}(x - y) g(y)\, \rmd y\\
					          &= \sum_{\nu} \int_{\T_\nu\cap B_\varepsilon(x)} \widetilde{\phi}_{\ep}(x - y) \nabla_y g(y)\, \rmd y = \sum_{\nu} \int_{\T_\nu\cap B_\varepsilon(x)} \widetilde{\phi}_{\ep}(x - y) w_{\nu}\, \rmd y = \sum_{\nu} \alpha_{\nu}^{x,\ep} w_{\nu}.%
\end{align*}
\end{proof}

\begin{lemma}\label{X4.13}
Given the same assumption as in Theorem \ref{X4.12}, let $\delta>0$. Then there exist smooth $P_\ep:(\cD_\cT)_{-\delta} \to \R^{n\times n}$ and $V_\ep,\mu_\ep:(\cD_\cT)_{-\delta} \to \R$, such that for every $x\in(\cD_{\cT})_{-\delta}$ we have%
\begin{equation}\label{LMIauxX}
  A_\ep(x) - \mu_\ep(x) P_\ep(x)\preceq \delta I%
\end{equation}
and%
\begin{equation}\label{mainQ}
  \dot{V}_\ep(x) + \widetilde m \mu_\ep(x) \le Q + \delta,%
\end{equation}
where%
\begin{equation*}
  A_\ep(x) := P_\ep(x)\rmD f(x) + \rmD f(x)\trn P_\ep(x) + \dot{P}_\ep(x)%
\end{equation*}
and $Q$ is the same constant as in Theorem \ref{X4.12}.
\end{lemma}

\begin{proof}
Let $P$, $V$, and $\mu$ be defined as in Theorem \ref{X4.12} and set%
\begin{equation*}
  G := \max\left\{\max_{\T_\nu\in\cT \atop i,j=1:n} \|w^\nu_{ij}\|_2,\max_{\xi\in\cT^*}\|\nabla V_\xi\|_2\right\}.%
\end{equation*}
Fix $0<\varepsilon<\delta$ so small that for all $x \in (\cD_\cT)_{-\delta}$ and all $y$ satisfying $\|x-y\|_2 < \varepsilon$ we have%
\begin{align*}
  |(B_\varepsilon)_{ij}(x)-B_{ij}(x)| &< \frac{\delta}{3n},\quad |B_{ij}(x)-B_{ij}(y)| < \frac{\delta}{3n},\\
  \|f(x)-f(y)\|_2 &< \frac{\delta}{3nG},\quad  |V(x)-V(y)| < \frac{\delta}{2},\\
  |\mu(x)-\mu(y)| &<\frac{\delta}{3\widetilde m},\quad |\mu_\varepsilon(x)-\mu(x)| < \frac{\delta}{3},%
\end{align*}
where the mollified functions with $\varepsilon$ in the subscript are defined as in Lemma \ref{cpasmoothlemma},%
\begin{equation*}
  B(x) := P(x)\rmD f(x) + \rmD f(x)\trn P(x) - \mu(x) P(x)%
\end{equation*}
and%
\begin{equation*}
  B_\varepsilon(x) := P_\varepsilon(x)\rmD f(x) + \rmD f(x)\trn P_\varepsilon(x) - \mu_\varepsilon(x) P_\varepsilon(x).%
\end{equation*}
Fix $x\in (\cD_\cT)_{-\delta}$. For each $\alpha_\nu^{x,\varepsilon}>0$, cf.~Lemma \ref{cpasmoothlemma}, for a $\T_\nu\in\cT$ select an $x_\nu$ in the interior of $\T_\nu\cap B_\varepsilon(x)$ and for each $\alpha_\xi^{x,\varepsilon}>0$ for a $\T_\xi\in\cT^*$ select an $x_\xi$ in the interior of $\T_\xi\cap B_\varepsilon(x)$.%

Now for all $i,j=1:n$ we have by the estimates above, Lemma \ref{cpasmoothlemma}, and the Cauchy-Schwarz inequality%
\begin{align*}
 (B_\varepsilon)_{ij}(x)+ \nabla (P_\varepsilon)_{ij}(x)\cdot f(x)
																	&< B_{ij}(x) + \frac{\delta}{3n} +\sum_{\nu} \alpha_{\nu}^{x,\varepsilon}  w^{\nu}_{ij} \cdot f(x) \\
&<\sum_{\nu} \alpha_{\nu}^{x,\varepsilon} \left(B_{ij}(x_\nu)+\frac{\delta}{3n}+w^{\nu}_{ij} \cdot f(x_\nu)+ \|w^{\nu}_{ij}\|_2 \frac{\delta}{3nG}\right)+ \frac{\delta}{3n} \\
&<\sum_{\nu} \alpha_{\nu}^{x,\varepsilon} \left(B_{ij}(x_\nu)+w^{\nu}_{ij} \cdot f(x_\nu)\right)+ \frac{\delta}{n}.%
\end{align*}
Hence,%
\begin{align*}
A_\varepsilon(x)-\mu_\varepsilon(x) P_\varepsilon(x) &= B_\varepsilon(x)+\dot{P}_\varepsilon(x)
\preceq  \sum_{\nu} \alpha_{\nu}^{x,\varepsilon}\left(B(x_\nu)+\dot{P}(x_\nu)\right) + \frac{\delta}{n} (1)_{ij} \\
&= \sum_{\nu} \alpha_{\nu}^{x,\varepsilon}\left(A(x_{\nu}) - \mu(x_{\nu})P(x_{\nu})\right) + \frac{\delta}{n} (1)_{ij} \preceq \delta I.%
\end{align*}
Similarly,%
\begin{align*}
 \dot{V}_\varepsilon(x)+\widetilde m \mu_\varepsilon(x) &= \sum_{\xi} \alpha_{\xi}^{x,\varepsilon}\left( \nabla V_\xi\cdot f(x) + \widetilde m \mu(x)\right) + \frac{\delta}{3}\\
 &= \sum_{\xi} \alpha_{\xi}^{x,\varepsilon}\left( \nabla V_\xi\cdot f(x_\xi) + \widetilde m \mu(x_\xi)\right) + \frac{\delta}{3n}+\frac{\delta}{3}+\frac{\delta}{3} \le Q+\delta.%
\end{align*}
\end{proof}

Note that the Constraints 1 and 2 in Optimization Problem \ref{SDP2} are not very strongly coupled to the Constraints 3 and 4. Constraints 2 balance the values $\mu(x_k)$ and the gradient $D_\xi^\mu$, whereas Constraints 4 do not have to take the gradient of $\mu$ into account. Since the gradient is multiplied by $h_\xi^2$, which is small for small simplices $\T_\xi\in \cT^*$, the gradient can be rendered less important in Constraints 2 by using smaller simplices. It is thus tempting to split Optimization Problem \ref{SDP2} into two optimization problems, the first with Constraints 1 and 2 and some objective function that makes the collection of the $\mu(x_k)$ small in some sense, and then consecutively run an optimization problem with Constraints 3 and 4, where the $\mu(x_k)$ from a solution to the first optimization problem are constants. This is especially tempting, because SDP solvers have not reached the maturity of linear programming solvers and are sometimes not able to deliver solutions to moderately sized feasible problems or worse, deliver solutions that are quite far from being feasible.%

Further, if we use Optimization Problem \ref{SDP1} to find a constant matrix $P(x)$, the optimization problem is much smaller and easier to solve. From such a solution the Optimization Problem \ref{SDP2} can be naturally split into two optimization problems as described above without any disadvantage, because the coupling between Constraints 1 and 2 on the one hand and Constraints 3 and 4 on the other hand vanishes completely. Even better, since $P(x)$ is constant, its gradient is zero and therefore $D_\nu=0$. Thus, the gradient of $\mu$ plays no role, because its upper bound $D_\xi^\mu$ is multiplied by $D_\nu$ in Constraints 2. We can thus drop Constraints 1 and compute the optimal $\mu(x_k)$ directly.%

First, for each vertex $x_k$ we find the minimum $\mu(x_k)$ such that \eqref{LMIaux} is fulfilled for every $\nu$ such that $x_k$ is a vertex of $\T_\nu$. Then we minimize $Q$ under the linear constraints \eqref{LPaux}. This is described in more detail in the next section.%

\subsection{Simplified procedure}\label{simpproc}

In the simplified procedure we restrict our search for a matrix $P(x)$ in Optimization Problem \ref{SDP1} to a constant matrix and then split Optimization Problem \ref{SDP2} into two simpler problems. In detail:%

In Optimization Problem \ref{SDP1} we set $P_{ij}(x_k) := P_{ij}$ and then $P(x_k) := (P_{ij})$ for all vertices $x_k$ of all simplices of $\cT$. Then clearly we can set $D_\nu:=0$ and $C:=C_\nu$ for all $\T_\nu \in \cT$ and the constraints simplify to:%
\begin{align*}
  &\epsilon_0 I\preceq P \preceq CI%
\end{align*}
and for each simplex $\T_\nu=\co(x_0,\ldots,x_n)\in\cT$ and each vertex $x_k$ of $\T_\nu$:%
\begin{align*}
  0 \succeq& A(x_k)- \mu P + h_\nu^2 \cdot 2n^3B_{3,\nu}C I,%
\end{align*}
where%
\begin{equation*}
  A(x_k) = P\rmD f(x_k) + \rmD f(x_k)\trn P,%
\end{equation*}
because $w^\nu_{ij}$, the gradient of $P$, is now the zero vector.%

Let us now consider Optimization Problem \ref{SDP2}. The Constraints 2 become{\rm :} For each simplex $\T_\xi=\co(x_0,\ldots,x_n)\in\cT^*$ and each vertex $x_k$ of $\T_\nu$:%
\begin{equation}\label{newA}
  0 \succeq A(x_k) - \mu(x_k) P + h_\xi^2\cdot 2n^3B_{3,\nu}C I,%
\end{equation}
because $D_\nu=0$. The variables $D^\mu_\xi$ are thus redundant and we can eliminate Constraints 1. An even farther reaching consequence is that we do not even have to combine the Constraints \eqref{newA} with Constraints 3 and 4.  We can compute the optimal $\mu(x_k)$ locally for each vertex $x_k$ by solving{\rm :}\\
For each vertex $x_k$ of the triangulation $\cT^*$ maximize the value $\mu(x_k)$ under the constraints%
\begin{equation*}
  0 \succeq A(x_k) - \mu(x_k) P + h_\xi^2\cdot 2n^3B_{3,\nu}C I%
\end{equation*}
for every $\nu$ such that $x_k\in\T_\nu\in\cT$.%

An even simpler version of this optimization problem is obtained by defining%
\begin{equation*}
  B_{3}^y := \max_{\T_\nu\in\cT}B_{3,\nu}%
\end{equation*}
and solving: For each vertex $x_k$ of the triangulation $\cT^*$ maximize the value $\mu(x_k)$ under the constraints%
\begin{equation*}
  0 \succeq A(x_k)- \mu(x_k) P + h_\xi^2\cdot 2n^3B_{3}^{x_k}C I.%
\end{equation*}

For small $h_\xi>0$ good estimates on these optimal $\mu(x_k)$ for the Optimization Problem \ref{SDP2} can be directly computed using \eqref{LOKAX}. Just set%
\begin{equation*}
%\mu(x_k)=\lambda_{\max}(x_k)+\frac{h_\xi^2\cdot 2n^3B_{3}^{x_k}C}{\|M^{-1}\|_2},
  \mu(x_k) := \lambda_{\max}(x_k)+h_\xi^2\cdot 2n^3B_{3}^{x_k}C\|P^{-1}\|_2,%
\end{equation*}
where $\lambda_{\max}(x_k)$ is the largest generalized eigenvalue of the matrix pair $(A(x_k),P)$. Since $C\ge \|P\|_2$, this formula can be further simplified to%
\begin{equation*}
  \mu(x_k) := \lambda_{\max}(x_k) + h_\xi^2\cdot 2n^3B_{3}^{x_k} \kappa_2(P),%
\end{equation*}
using the condition number $\kappa_2(P) := \|P\|_2\|P^{-1}\|_2$ of $P$.%

After this being done, we can minimize $Q$ under the Constraints 3 and 4 of Optimization Problem \ref{SDP2}, and this is a linear programming problem, for which much more mature solvers exist.%

\section{An example: the Lorenz system}\label{sec_lorenz}

We consider the Lorenz system%
\begin{equation}
\label{lorenzsys}
\frac{\rmd}{\rmd t}\begin{pmatrix}
              x \\
              y \\
              z \\
            \end{pmatrix}=  \begin{pmatrix}
                              -\sigma x +\sigma y \\
                              rx-y -xz \\
                              -bz+xy \\
                            \end{pmatrix}
            =:g(x,y,z).%
\end{equation}
For our approach it is advantageous to scale the system such that its attractors are contained in a smaller set. For this purpose, define $S:=\diag(s_x,s_y,s_z)$ for constants $s_x,s_y,s_z>0$ and consider the system $\dot \bx = f(\bx)$ with $f(\bx)=S^{-1}g(S\bx)$, i.e.\ the system%
\begin{equation}
\label{slorenz}
\frac{\rmd}{\rmd t}\begin{pmatrix}
              x \\
              y \\
              z \\
            \end{pmatrix}=  \begin{pmatrix}
                              -\sigma x + \sigma \frac{s_y}{s_x} y \\
                              r\frac{s_x}{s_y}x-y -\frac{s_xs_z}{s_y}xz \\
                              -bz+\frac{s_xs_y}{s_z}xy \\
                            \end{pmatrix}.
\end{equation}
Clearly one can take%
\begin{equation*}
  B_\nu = s_x\cdot \max\left\{\frac{s_y}{s_z},\frac{s_z}{s_y}\right\}
\end{equation*}
in the optimization problems for all $\T_\nu$, independent of the triangulation $\cT$, because the right-hand side is a global bound on the second-order derivatives of $f$, and  we can set $B_{3,\nu} := 0$ for all $\T_\nu$, because the components of $f$ are second-order polynomials. We use the scaling parameters $s_x=24.5$ and $s_y=s_z=100$ and thus $B_\nu = 24.5$ in what follows.%

In \cite[\S2.2]{BLR} it is shown that if $\sigma\ge 1$ and $b \ge 2$ in \eqref{lorenzsys}, then the system is dissipative in the sense of Levinson and the region $\cD$ of dissipation fulfills%
\begin{align}
\cD&\subset \left\{(x,y,z)\in\R^3\,:\,x^2+y^2+(z-[\sigma+r])^2 \le \frac{b}{2}(\sigma+r)^2\right\}, \label{DTOI}\\
\cD&\subset \left\{(x,y,z)\in\R^3\,:\,2x^2+y^2+(z-[\sigma+r])^2 \le \left[1+\frac{(b-2)^2}{4(b-1)}\right](\sigma+r)^2\right\},\label{DTOII}\\
\cD&\subset \left\{(x,y,z)\in\R^3\,:\, y^2+(z-r)^2 \le \frac{b^2r^2}{4(b-1)}\right\}, \label{DTOIII}  \\
\cD&\subset \left\{(x,y,z)\in\R^3\,:\, z \ge 0\right\} \label{DTOIV}.
\end{align}
All attractors of \eqref{lorenzsys} are inside $\cD$. For the common parameters $\sigma=10$, $r=28$, and $b=8/3$ in \eqref{lorenzsys}, it suffices to compute our metric $P$ and Lyapunov-type function $V$ on the set%
\begin{equation*}
  \cK:=[-1,1]\times[-0.29,0.29]\times [0,0.57],%
\end{equation*}
because with these parameters \eqref{DTOIII} implies%
\begin{equation*}
  |y|\le \frac{4}{\sqrt{15}}r \le 29 = s_y\cdot 0.29%
\end{equation*}
and \eqref{DTOIII} and \eqref{DTOIV} imply%
\begin{equation*}
  0\le z \le \left(1+\frac{4}{\sqrt{15}}\right)r \le 57 = s_z\cdot 0.57,%
\end{equation*}
which in turn with \eqref{DTOII} implies that%
\begin{equation*}
 |x|\le \sqrt{\frac{1}{2}\left(\frac{16}{15}(\sigma+r)^2 - \left(\frac{4}{\sqrt{15}}r-\sigma\right)^2\right)} \le 24.5 =s_x \cdot 1.%
\end{equation*}
Thus, all the attractors of \eqref{slorenz} are inside of $\cK$.%
%the \it{butterfly} of the Lorenz system is mapped into the cube $\cK:= [-1,1]\times [-0.125,0.125]\times [0,0.5]$ by fixing $s_x=0.05$, $s_y=s_z=0.01$, and $\theta=-\pi/18$, cf.~Figure \ref{lorenzh} where we plotted the solution trajectories of \eqref{slorenz} in $\cK$
%for a few different initial points.
%\begin{figure}[!h]
%  \centering
%    \includegraphics[height=6cm]{figlorenzh.jpg}
%      \caption{The trajectories of system \eqref{slorenz} for a few different initial values. The {\it butterfly} of the Lorenz system can be clearly seen, }\label{lorenzh}
%\end{figure}
%These are the parameter we will use and we will compute the metric $P$ and the Lyapunov-type function $V$ on the cube $\cK$.  Strictly, one should compute them on a forward invariant or an absorbing set.  However, available
%upper bounds for absorbing sets are quite conservative, e.g.~
%$$
%\{x^2+y^2+ zz \} \cap \{z \ge 0\}
%$$
%in DIMENSION THEORY BOOK, and the optimization problems become very large if we use such a set as a domain for $P$ and $V$.
% (10 cores, 3.3 GHz) and 128 GB RAM

For our example we used the simplified procedure from Section \ref{simpproc} and the computer we used has an i9-7900X CPU.

\begin{remark}
The implementation of our algorithm, even the simplified one, is  not simple, and a detailed discussion of it is beyond the scope of this paper.  We refer the reader to \cite{HaCpp1,HaCpp2} for some implementation details on triangulations for the computation of Lyapunov functions. We used similar methods, adapted to our problem.
\end{remark}
% XXX main comment 6

First we compute a constant metric $P(x)=P$ on the set $\cK$ using a triangulation with the vertices%
\begin{equation*}
  \left(1.0\cdot \frac{i_x}{12}, 0.29\cdot \frac{i_y}{6},0.57\cdot \frac{i_z}{10}\right)\ \ \text{for $i_x=-12:12$, $i_y=-6:6$, and $i_z=0:10$.}
\end{equation*}
The triangulation we used is a so-called {\em standard-triangulation} as in \cite[\S4]{GH3}, but with different scaling along the different axes, i.e.\ $\rho$ in \cite[Def.~4.8]{GH3} is $\rho_x=1/12$ along the $x$-axis, $\rho_y=0.29/6$ along the $y$-axis, and $\rho_z=0.57/10$ along the $z$-axis.
%We compute a constant metric $P(x)=M$ and can therefore use the simplified optimization problems.  It is not difficult, but tedious, to show that we can take
%$$
%B_\nu = \frac{s_xs_z}{s_y}\max\{|\sin2\theta|,\cos^2\theta,\sin^2\theta\}
%$$
%for all $\T_\nu$, because the right-hand side is a global bound on the second-order derivatives of $f$.  It is simple to see that we can set
%$B_{3,\nu}=0$ for all $\T_\nu$, because there the components of $f$ are second-order polynomials.  Both of these bounds are clearly independent of the triangulation $\cT$.
%WEIL WIR SO EINFACHE TRIANGULIERUNG NEHMEN KANN MAN DIESE GRENSEN EIN BISSCHEN BESSER MACHEN, SEHE DISS UND/ODER DTSA PAPER.  EVTL ERWAEHNEN?
We set $\epsilon_0:=0.1$ in Optimization Problem \ref{SDP1} and the constraints become%
\begin{align*}
  0.1 I \preceq P \preceq CI%
\end{align*}
for each simplex $\T_\nu\in\cT$ and for each vertex $x_k$ of $\T_\nu$:%
\begin{align*}
  0 \succeq& P\rmD f(x_k) + \rmD f(x_k)\trn P - \mu P.%
\end{align*}
The optimization problem is especially simple because $B_{3,\nu}=0$ for all $\T_\nu$, and therefore we do not even have to repeat the second constraints for all $\T_\nu$, i.e.\ we can replace {\em for each simplex $\T_\nu\in\cT$ and for each vertex $x_k$ of $\T_\nu$} with {\em for each vertex $x_k$ of a simplex in $\cT$}.%

By trying out a few different $\mu$s, we obtained a feasible solution with $\mu=27$. Investigation gave us that there is only one positive generalized eigenvalue for all $x$, so we can take $\widetilde m=1$. Writing the problem took a few seconds with our software and solving the optimization problem took 140 sec.~using the solver PENSDP 2.2 \cite{SPDsol}. The problem has 7 variables and 69,122 matrix constraints. The matrix computed is%
%PENSDP 2.2
%8 variables, constr =0, mconstraints=7682, mmaxsize=3, top ent=22.3618,
%C  =  0.3826528042755
%M 0 0 0 = 0.1028221182184
%M 0 0 1 =  0.0281153794808
%M 0 0 2 = 0
%M 0 1 1 = 0.3798288459553
%M 0 1 2 =0
%M 0 2 2 =  0.2690272058456
%$$
%M=\begin{pmatrix}
%    0.102822 & 0.028115 & 0 \\
%    0.028115 & 0.379829 & 0 \\
%    0 & 0 & 0.269027 \\
%  \end{pmatrix}
%$$
$$ %0.1008469737786 << -0.01415360101927 << -6.397916306823e-17 << endr << -0.01415360101927 << 0.3361537095909 << 8.886062745026e-17 << endr << -6.397916306823e-17 << 8.886062745026e-17 << 0.3139832543019;
P=\begin{pmatrix}
    0.1008469737786 &   -0.01415360101927 &   0\\
   -0.01415360101927 &   0.3361537095909 &  0\\
   0 & 0 &  0.3139832543019
  \end{pmatrix}
$$
which has $\|P^{-1}\|_2=0.1$ as the smallest and $C=\|P\|_2=0.3370007906512$ as the largest eigenvalue.
%WIE WAERE DIE P MATRIX IN DEM PROGROMSKY PAPER IN UNSEREN KOORDINATEN?
If we only used these results in the formula in \cite[Thm.~3.2]{PMa} for the upper bound on the topological/restoration entropy, i.e.\ set $V(x)=$const., this $P$ delivers $\lambda/(2\ln(2))\approx 19.4764$ as an upper bound.%

The formula (13) in \cite{PMa} delivers the upper bound%
\begin{equation}\label{PMest}
  \frac{1}{2\ln(2)}\left(\sqrt{(\sigma-1)^2+4r\sigma}\, -(\sigma+1)\right) \approx 17.0638
\end{equation}
for our parameters.%

For computing the Lyapunov-type function we used the triangulation $\cT^*$, which is constructed exactly as the triangulation $\cT$ above, but with the vertices%
\begin{equation}\label{gridform}
  \left( 1.0\cdot \frac{i_x}{N_x}, 0.29\cdot \frac{i_y}{N_y},0.57\cdot \frac{i_z}{N_z}\right)\ \ \text{for $i_x=-N_x:N_x$, $i_y=-N_y:N_y$, and $i_z=-1:N_z$,}
\end{equation}
for some $N_x,N_y,N_z \in \N$. We tried a few different sets of parameter values, expecting lower upper bounds on the topological/restoration entropy for larger values of $N_x$, $N_y$, and $N_z$. We used the state of the art solver GUROBI, which is free for academic use, to solve the LP problems using the barrier method.%

The first set of parameters was $N_x=30$, $N_y=14$, and $N_z=28$ and with those the LP problem with 930,032 variables and 3,800,122 constraints was written in 11 sec.~with our software and solved in 291 sec.~with the optimal value of $Q=23.9094$ which delivers the upper bound $17.247$, which is slightly worse than in \eqref{PMest}.%

The second set of parameters was $N_x=42$, $N_y=14$, and $N_z=28$ and with those the LP problem with 1,301,696 variables and 5,320,176 constraints was written in 28 sec.~with our software and solved in 889 sec.~with the optimal value of $Q=23.5254$ which delivers the upper bound $16.970$, which is slightly better than in \eqref{PMest}.%

Because we are using such a simple axially parallel triangulation, one can use somewhat less conservative bounds the LP problems.  That is, the term $nh_\xi^2B^*_\xi$ in Constraints 4 in Optimization Problem \ref{SDP2} can be replaced with a smaller number and Theorem \ref{X4.12} still holds true. For these less conservative bounds we refer to \cite[Lem.~4.16]{Mar}. Using these less conservative bounds in the LP problems gave notably better results. Using the first set of parameters, the LP problem was solved in 437 sec.~with the optimal value of $Q=22.5403$, which delivers the upper bound $16.260$, and using the second set of parameters the LP problem was solved in 508 sec.~with the optimal value of $Q=22.094$, which delivers the upper bound $15.937$.%

For the third set of parameters we took $N_x=50$, $N_y=18$, and $N_z=32$ and only used the less conservative bounds on the second-order derivatives of $f$. The LP problem had 2,265,460 variables and 9,266,354 constraints, was written in 24 sec.~with our software and solved in 877 sec.~with the optimal value of $Q=21.701$ which delivers the upper bound $15.654$.%

The fourth and final set of parameters was $N_x=70$, $N_y=22$, and $N_z=40$ and only used the less conservative bounds on the second-order derivatives of $f$. The LP problem had 4,812,572 variables and 19,699,632 constraints, was written in 61 sec.~with our software and solved in 5,801 sec.~with the optimal value of $Q=21.311$ which delivers the upper bound $15.373$.%

Thus, the best estimate we got with our method was the upper bound $15.373$ on the topological/restoration entropy, which is considerably better than $17.064$ given by formula (13) in \cite{PMa}.%
The results of the computations are summarized in Table \ref{tabresults}.%
\begin{table}
\begin{center}
\begin{tabular}{|c|c|c|c|c|c|c|}
  \hline
 $N_x$ & $N_y$ & $N_z$ & time\,[s] &  impr.~bounds & $Q$ & u.b. \\
 \hline
  30 & 14 & 28 & 302 & No & 23.909 & 17.247 \\
  30 & 14 & 28 & 448 & Yes & 22.540 & 16.260 \\
  42 & 14 & 28 & 917 & No & 23.525 &16.970 \\
  42 & 14 & 28 & 536 & Yes & 22.094 & 15.937 \\
  50 & 18 & 32 & 901 & Yes & 21.701 & 15.654 \\
  70 & 22 & 40 & 5862 & Yes & 21.311 & 15.373 \\
  \hline
\end{tabular}
\caption{The results of our computations. $N_x,N_y,N_z$ are the parameters for the grid in \eqref{gridform}, `time' is the total time in seconds needed to write and solve the problem, `impr.~bounds' states whether the improved bounds discussed in the text are used (Yes) or not (No), $Q$ is the objective that is minimized in Optimization Problem \ref{SDP2}, and `u.b.' is the associated upper bound on the topological/restoration entropy. For reference, the upper bound 17.064 is computed in \cite{PMa}.}
\label{tabresults}
\end{center}
\end{table}

Running the full Optimization algorithms \ref{SDP1} and \ref{SDP2} and thus computing a nonconstant matrix $P$ would be very interesting, but is hardly possible for examples of interest with today's SDP problems solvers. It remains interesting to see if these solvers mature enough in the near future for this to change and how much the upper bound decreases using our fully fledged method.%

\section{Concluding remarks}\label{sec_concl}

In this paper, we proposed an algorithm for computing upper bounds on the critical channel capacity for state estimation over a finite-capacity channel, a typical problem studied in networked control. The upper bounds computed by our algorithm are, at the same time, upper bounds on the topological entropy of the dynamical system under consideration. Moreover, the output of the algorithm can be used to implement a coding and estimation policy which operates over a channel of the corresponding capacity.%

It is not hard to see that topological entropy, in general, cannot be approximated very well by our algorithm, since the computed values are upper bounds on restoration entropy $h_{\res}$,
as shown in Section \ref{sec_stateest}, and the strict inequality $h_{\tp} < h_{\res}$ holds for most dynamical systems. Hence, we do not claim that our paper contributes to the problem of numerical computation of topological entropy. Some standard references on this quite intricate subject (for multi-dimensional systems) are \cite{COH,Dea,FJO,NPi}.%

At this point, it is not clear whether restoration entropy can be approximated (and not only upper-bounded) by the estimates of Theorem \ref{thm_matpog}. We believe, however, that this is the case and hope to deliver a proof in a future work.%

Our full algorithm, using both Optimization Problem \ref{SDP1} and Optimization Problem \ref{SDP2} for a non-constant matrix $P$, overstrains currently even the  best semidefinite-programming solvers in problems of interest. It will be interesting to see if this situation changes in the near future. Therefore, we derived a simplified algorithm in Section \ref{simpproc}, which computes an then uses a constant $P$.  Using this simplified algorithm allowed us to study the Lorenz system with our method and we got superior results to \cite{PMa}, where an analytical bound is derived. It remains an interesting question how much the full algorithm can improve these bounds.%


\begin{thebibliography}{99}
\bibitem{BGH} R.~Baier, L.~Gr\"{u}ne, S.~Hafstein. \emph{Linear programming based Lyapunov function computation for differential inclusions}. Discrete Contin.\ Dyn.\ Syst.\ Ser.\ B 17(1) (2012), 33--56.%
\bibitem{BLR} V.~Boichenko, G.~Leonov, V.~Reitmann. \emph{Dimension Theory for Ordinary Differential Equations}. Teubner, 2005.
\bibitem{COH} Q.~Chen, E.~Ott, L.~Hurd. \emph{Calculating topological entropies of chaotic dynamical systems}. Phys.\ Lett.\ A 156 (1991), no.\ 1--2, 48--52.%
\bibitem{Dea} G.~D'Alessandro, P.~Grassberger, S.~Isola, A.~Politi. \emph{On the topology of the H\'enon map}. J.\ Phys. A 23 (1990), no.\ 22, 5285--5294.%
\bibitem{FJO} G.~Froyland, O.~Junge, G.~Ochs. \emph{Rigorous computation of topological entropy with respect to a finite partition}. Phys.\ D 154 (2001), no.\ 1--2, 68--84.%
\bibitem{GH1} P.~Giesl, S.~Hafstein. \emph{Construction of a CPA contraction metric for periodic orbits using semidefinite optimization}. Nonlinear Anal.\ 86 (2013), 114--134.%
\bibitem{GH2} P.~Giesl, S.~Hafstein. \emph{Revised CPA method to compute Lyapunov functions for nonlinear systems}. J.\ Math.\ Anal.\ Appl.\ 410 (2014), 292--306.%
\bibitem{GH3} P.~Giesl, S.~Hafstein. \emph{Computation and Verification of Lyapunov Functions}. SIAM J.\ Appl.\ Math.\ 14 (4) (2015), 1663--1698.%
\bibitem{HaCpp1} S.~Hafstein. \emph{Implementation of Simplicial Complexes for CPA functions in C++11 using the Armadillo Linear Algebra Library}. In Proceedings of 4th International Conference on Simulation and Modeling Methodologies, Technologies and Applications (SIMULTECH), Reykjavik, Iceland, 2013, 49--57.%
\bibitem{HaCpp2} S.~Hafstein. \emph{Efficient Algorithms for Simplicial Complexes Used in the Computation of Lyapunov Functions for Nonlinear Systems}. In Proceedings of 7th International Conference on Simulation and Modeling Methodologies, Technologies and Applications (SIMULTECH), Madrid, Spain, 2017, 398--409.%
\bibitem{KYu} C.~Kawan, S.~Y\"{u}ksel. \emph{On optimal coding of non-linear dynamical systems}. IEEE Trans.\ Inform.\ Theory 64 (2018), no.~10, 6816--6829.%
\bibitem{SPDsol} M.~Kocvara, M.~Stingl. \emph{PENNON: Software for Linear and Nonlinear Matrix Inequalities. In: Handbook on Semidefinite, Conic and Polynomial Optimization}. Springer 2012, 755--794.%
\bibitem{LMi} D.~Liberzon, S.~Mitra. \emph{Entropy and minimal bit rates for state estimation and model detection}. IEEE Trans.\ Automat.\ Control 63 (2018), no.~10, 3330--3344.%
\bibitem{Mar} S.~Marin\'osson. \emph{Stability Analysis of Nonlinear Systems with Linear Programming: A Lyapunov Functions Based Approach}. PhD thesis: Gerhard-Mercator-University Duisburg, Duisburg, Germany, 2002.%
\bibitem{MP1} A.~Matveev, A.~Pogromsky. \emph{A topological entropy approach for observation via channels with limited data rate}. IFAC Proceedings Volumes 44.1 (2011), 14416--14421.%
\bibitem{MP2} A.~Matveev, A.~Pogromsky. \emph{Observation of nonlinear systems via finite capacity channels: constructive data rate limits}. Automatica J.\ IFAC 70 (2016), 217--229.%
\bibitem{MP3} A.~Matveev, A.~Pogromsky. \emph{Observation of nonlinear systems via finite capacity channels. Part II: Restoration entropy and its estimates}. Submitted, 2017.%
\bibitem{NPi} S.~Newhouse, T.~Pignataro. \emph{On the estimation of topological entropy}. J.\ Statist.\ Phys.\ 72 (1993), no.\ 5--6, 1331--1351.%
2016).%
\bibitem{PMa} A.~Pogromsky, A.~Matveev. \emph{Estimation of topological entropy via the direct Lyapunov method}. Nonlinearity 24 (2011), no.\ 7, 1937--1959.%
\bibitem{Sav} A.~V.~Savkin. \emph{Analysis and synthesis of networked control systems: topological entropy, observability, robustness and optimal control}. Automatica J.\ IFAC 42 (2006), no.\ 1, 51--62.%
\end{thebibliography}
\end{document}